\documentclass[12pt]{article}
\usepackage{amssymb, latexsym,amsmath}

\usepackage{graphics}
\usepackage{graphicx,rotating}
\usepackage{pstricks}
\usepackage{pst-node}

\newtheorem{theorem}{Theorem}[section]
\newtheorem{proposition}[theorem]{Proposition}
\newtheorem{lemma}[theorem]{Lemma}
\newtheorem{claim}[theorem]{Claim}
\newtheorem{definition}[theorem]{Definition}
\newtheorem{fact}[theorem]{Fact}
\newtheorem{corollary}[theorem]{Corollary}
\newtheorem{remark}[theorem]{Remark}

\newcommand{\qed}{\relax\ifmmode\eqno\B ox\else\mbox{}\quad\nolinebreak\hfill$\Box$\smallskip\fi}

\newenvironment{proof}{\noindent{\it Proof} :}{\qed}

\newenvironment{prclaim}{\noindent{\it Proof of the Claim} :}{\qed}

\newcommand{\nc}{\newcommand}

\newcommand{\Kpinf}{K^{p^{\infty}}}

\nc{\K}{K[X_{\infy}]}
\nc{\KKn}{K[X_{1 \leq n},...,X_{k \leq n}]}

\nc{\KKK}{K[X_{1\infty},\ldots,X_{k\infty}]}
\nc{\KxK}{K[x_{1\infty},\ldots,x_{k\infty}]}

\nc{\Iacf}{I_{acf}}
\nc{\scftype}{(scf)-type\  }
\nc{\acftype}{(acf)-type\ }
\nc{\scftypes}{(scf)-types\ }
\nc{\acftypes}{(acf)-types\ }
\nc{\tacf}{t_{acf}}
\font\help=cmr7
\def\strsubset{\hbox{$\subseteq\kern-.8em\lower.2em\hbox{\help
\char'57}$}\,}
\def\dnfo{\,\raise.2em\hbox{$\,\mathrel|\kern-.9em\lower.35em\hbox
{$\smile$}$}}
\def\dnf#1{\lower.9em\hbox{$\buildrel\dnfo\over{  #1}$}}
\def\dfo{\;\raise.2em\hbox{$\mathrel|\kern-.9em\lower.35em\hbox{$\smile$}
\kern-.7em\hbox{\char'57}$}\;}
\def\df#1{\lower.9em\hbox{$\buildrel\dfo\over{ #1}$}}

\newcommand{\Fp}{\mathbb{F}_p}

\newcommand{\Gm}{\mathbb{G}_m}

\newcommand{\U}{{\cal U}}

\title{On function field Mordell-Lang: the semiabelian case and the socle theorem}

\author{Franck Benoist\thanks{Partially supported by  ANR ValCoMo
    (ANR-13-BS01-0006) Valuations, Combinatorics and Model Theory}
   \and Elisabeth
  Bouscaren\footnotemark[1] \and Anand Pillay\thanks{Partially supported by  NSF grant DMS-1360702} }

\begin{document}

\maketitle

\begin{abstract} In this paper we complete our ``second generation" model-theoretic account of the function field Mordell-Lang conjecture, where
  we avoid appeals to dichotomy theorems for (generalized) Zariski geometries.  In the current paper we reduce the semiabelian case to the abelian case using model-theoretic tools. We also use our results from \cite{BBP2} to prove modularity of the ``Manin kernels"  (over $\Fp(t)^{sep}$ in positive characteristic and over ${{\mathbb C}(t)}^{alg}$  in characteristic $0$ ). 
\end{abstract}

\noindent 2010 Mathematics Subject Classification 03C45, 03C98 (primary), 11G10, 14G05 (secondary).

\section{Introduction}
This paper is the third in a series of papers by the authors, of which the first two are \cite{BBP1}, \cite{BBP2},  revisiting the model-theoretic approach to function field Mordell-Lang initiated by Hrushovski \cite{Hrushovski}, and where  positive characteristic is the main case of interest. In \cite{BBP1} we gave, among other things, a counterexample to a claim, implicit in \cite{Hrushovski},  that a key model-theoretic object $G^{\sharp}$ attached to a semiabelian variety $G$, had finite relative Morley rank.  In \cite{BBP2} and the current paper, the main aim is to give a model-theoretic proof of the main theorem (function field Mordell-Lang) avoiding the appeal in \cite{Hrushovski} to results around Zariski geometries.   Our reasons are (i)  this appeal to Zariski geometries is something of a black box, which is difficult for model theorists and   impenetrable for non model-theorists and (ii) in the positive characteristic case, it is ``type-definable" Zariski geometries which are used and for which there is no really comprehensive exposition, although the proofs in \cite{Hrushovski} are correct.  In \cite{BBP2} we concentrated on the abelian variety case, and  found a new model-theoretic proof when the  ground field $K$ is $\Fp(t)^{sep}$, by reduction to a Manin-Mumford statement. 
  The new contributions  of the current paper are the following:

  (i) We give a reduction of function field Mordell-Lang for semiabelian varieties to the case of traceless abelian varieties in all characteristics and without any assumptions on the ground field. This result,  together with \cite{BBP2},  gives a self-contained model theoretic proof of function field Mordell-Lang for semiabelian varieties over $\Fp(t)^{sep}$, by reduction to a Manin-Mumford statement (and avoids  appeal to Zariski geometries).

  (ii) We prove the ``weak socle theorem" from \cite{Hrushovski} but in the general form that is really needed both in \cite{Hrushovski} and in the current paper (namely the finite $U$-rank rather than finite Morley rank context).

  (iii) We introduce a new algebraic-geometric object, in all characteristics,  the algebraic $K_{0}$-socle of a semiabelian variety $G$ over $K_{1}$, where $K_{0}< K_{1}$ are algebraically closed fields. We use this socle to reduce function field Mordell-Lang for semiabelian varieties to the case of semiabelian varieties which are isogenous to a direct product of a torus and an abelian variety.  

  (iv) We give a new proof of the modularity (or $1$-basedness) of $A^{\sharp}$ when $A$ is simple with trace $0$ over the constants, for the cases from \cite{BBP2}, where we derived Mordell-Lang for abelian varieties over ${\Fp(t)}^{sep}$ or over ${{\mathbb C}(t)}^{alg}$ without appealing to the dichotomy theorem for Zariski geometries. 


Here is a brief description of what is covered, section-by-section. 
In Section 2 we state and prove the ``weak socle theorem" for type-definable groups of finite $U$-rank. In Section 3 we apply this to the case where the type-definable group is $G^{\sharp}$ for $G$ a semiabelian variety, and we also introduce the algebraic socle. In Section 4 we reduce Mordell-Lang to the case of the algebraic socle, and then prove the main result of the paper. In Section 5 we discuss model theoretic properties of the socle of $G^{\sharp}$, in particular proving quantifier elimination and finite relative Morley rank. In the appendix we give proofs of {\em modularity} (or $1$-basedness) of $A^{\sharp}$ as mentioned earlier.

 We assume  familiarity with model theory, basic stability, as well as
differentially and separably closed fields.  The book \cite{MMP} is a
reasonable reference, as well as \cite{Pillaybook} for more on
stability theory. Definability means with parameters unless we say
otherwise.  

We finish the introduction with some  more comments around the socle theorem. 
The socle is used, via the weak socle theorem  in both characteristics  by Hrushovski to
prove function field Mordell-Lang directly for semiabelian varieties. 
For an abelian variety $A$, the  socle of  $A^\sharp$ is $A^\sharp $ itself, but still the weak
socle theorem is needed in characteristic $0$ for our  model theoretic proofs for example in \cite{BBP2}. But this weak socle theorem is
{\em not} needed
for abelian varieties in characteristic $p$. 
It is needed though for semiabelian  varieties even in characteristic  $p$, and in
fact in this paper we aim to explain why by showing how the use of the
socle allows us to reduce  to semiabelian varieties of a particular 
pleasant form. 

We will see in particular that, in characteristic $p$,  the ``bad behavior'' from the point of
view of model theory of some semiabelian varieties $G$,  which translates into ``bad behavior''
of $G^\sharp$,   disappears when one replaces $G^\sharp$  by its socle.

\section{The model theoretic socle;  the abstract case}

\subsection{Preliminaries}

We suppose that $T = T^{eq}$ is a stable theory. We work in a big model of $T$, $\cal U$, saturated  and homogeneous. 

Some  conventions: Let $S\subset \mathcal{U}^n$, we say that $S$ is {\em type-definable} if $S$ is an intersection of  definable sets, or equivalently if $S$ is a partial type, that is, given by a conjonction of formulas  which can be infinite (of cardinality strictly smaller than that of $\mathcal U$).  As usual, we will identify a type-definable $S(x) $ with its set of realizations in the big  model $\cal U$. For $M\preceq \cal U$ a model of $T$, if $S$ is type-definable in $\mathcal{U}^n$, with parameters from $M$, we denote  by $S(M)$ the set of realizations of
$S$ in $M$, $ S(M) := \{m \in M^n : M \models S(m)\} $.

When we say that $S$ is a {\em type-definable   minimal set}, we mean that for any definable set $D$, $S(\U) \cap D(\U) $ is finite or cofinite in $S(\cal U)$. In other words, every relatively definable subset of $S(\U)$ is finite or cofinite. 
When we say that $Q$ is a {\em minimal (complete) type} over a given set of parameters $A$, we mean that $Q$ is minimal and is also a complete type over $A$, hence  is precisely a complete stationary type of $U$-rank one over $A$. 

\smallskip

Further conventions of notation: If $H, G$ are groups, we will use the notation $H<G$ to mean that $H$ is a (non necessarily proper) subgroup of $G$. Likewise, if $K_0, K_1$ are fields,  $K_0 < K_1$ just means that $K_0$ is a subfield of $K_1$.

\noindent We recall first the classic Indecomposability theorem, stated for the
case of a group of finite $U$-rank. This is, in principle, a particularly simple case
($\alpha = 0$) of the  Theorem of
$\alpha$-indecomposables for superstable groups, which generalizes  the
Zilber Indecomposability
Theorem for groups of finite Morley Rank (Theorem 6.10 in \cite{Poizatbook} or
Theorem 3.6.11 in \cite{Wagnerbook}).  However we are typically working with a type-definable  group $G$ rather than a definable group, so
strictly speaking there may not be an ambient superstable theory in which $G$ is type-definable. Nevertheless, the proofs in the references above easily adapt. 

\smallskip

\begin{definition} A  type-definable subset $S$ of $G$ is said to be {\em indecomposable} if for every relatively definable subgroup $H$ of $G$, either $S/H$ is infinite or all elements of $S$ are in the same coset of $H$ (denoted $|S/H|= 1$).\end{definition} 

Let us recall first: 
\begin{lemma}\label{stationaryindecomposable}
  -- If $Q$ is a complete stationary  type in $G$, then $Q$ is indecomposable. Moreover  if $a$ is any realization of the type $Q$ in $\cal U$, then the type-definable
  set $aQ:=\{ x \in {\cal U}: a^{-1}x \models Q\}$
  is also indecomposable.

-- Any type-definable connected subgroup of $G$ is indecomposable.
\end{lemma}
\begin{proof} Suppose $G$ and $Q$ are defined over $\emptyset$. Let $H_b$ be any relatively definable subgroup of $G$, defined over some finite set $b$ such that $|Q/H_b| = n >1$. Let $H_0$ be the intersection of all the $H_{b'}$, with $b' \in \U$, realizing the same type as $b$ over $\emptyset$. Then by stability of $T$ and the  Baldwin-Saxl condition (Prop. 1.4 in \cite{Poizatbook}), $H_0$ is the intersection of a finite number of the $H_b$.  Then $H_0$ is definable over $\emptyset$ and $Q/H_0$ is finite, strictly bigger than $1$. This contradicts the stationarity of $Q$. The rest is obvious. \end{proof}

\begin{proposition} {\em The Indecomposability Theorem}. \label{indecomposability} Let  $G$ be a type-definable  connected group of
finite $U$-rank, $n$. \\
Let $\{S_i :i \in I\} $ be a  family of type-definable indecomposable subsets of $G$, such that each $S_i$ contains  the identity of $G$. Then the subgroup $H$ of $G$ generated by
the $S_i$ is type-definable and connected.  In fact, there is
some $m\leq n$ such that $H = (S_{i_1} \ldots S_{i_m})^2$. 
\end{proposition}

\begin{corollary} \label{connected} 1. If $(B_i)_{i \in I} $  are type-definable connected
  subgroups of $G$, then the group generated by the $B_i$ is type-definable and connected. Furthermore it is the product of finitely
  many of the $B_i$. \\
\end{corollary}

\medskip 
From now on, we will suppose that the group $G$ is commutative, to simplify notation. 

\begin{definition} {\rm 1.} Let $G$ be a type-definable   commutative connected group of finite
  $U$-rank.
We say that $G$ is {\em almost generated} by type-definable sets $Q_1,\ldots, Q_n$
if $G \subset acl ( F \cup  Q_1\cup \ldots \cup Q_n)$, where $F$ is some finite 
set of parameters,  by which we mean that  
$\cal U$, $G({\cal U}) \subset acl ( F \cup Q_1({\cal  U} ) \cup \ldots \cup
Q_n({\cal U} )) $. 

When $G$ is definable and almost generated  by a single strongly minimal set, then $G$ is usually called {\em almost strongly minimal}. If $G$ is almost generated  by a single minimal type-definable set  $Q$, we say that $G$ is {\em almost minimal} with respect to $Q$. 
  In order to be able, further down,  to treat uniformly all characteristics, we will make the convention that we will say that $G$ is almost minimal with respect to $Q$, whether $Q$ is a strongly minimal definable set or a minimal type-definable set. 

If $G$ is almost generated  by minimal type-definable   $Q_{1},..,Q_{n}$ we say that $G$ is {\em almost pluri-minimal}. \\

\noindent {\rm 2.} Let $G$ be a type-definable  connected group,
defined over some $B =acl(B)$, $G$   is said to be {\em rigid} if,
passing to a saturated model, all connected type-definable subgroups of $G$ are defined over $acl(B)$. 
\end{definition}


\medskip

The following is well-known, but follows from the more general Lemma 4.7 from Chapter 7 of  \cite{Pillaybook}.

\begin{proposition} \label{nonorthogonal} Let $G$ be a type-definable  connected group,  almost minimal with respect to a minimal type $Q$.  If $Q$ is not  orthogonal to the type-definable set $D$, then there is a type-definable group $H \subset dcl (D) $ and a definable surjective homomorphism $h : G \to H$ with finite kernel.
\end{proposition}

\subsection{The Socle}

One can  find definitions and  proofs for the existence of the socle
and the ``weak socle theorem'' in the case of groups of finite Morley
Rank, for example in \cite{Hrushovski} or \cite{bouscaren}. But there is no  reference, as far as we know, 
for   the case of type-definable groups of finite $U$-rank, although it is a result which has been used before.  So we take this
opportunity to give a complete presentation here.  

Let again $T= T^{eq} $  be stable and $G$ be a type-definable commutative  group,   connected and of finite
$U$-rank, $n$. 
We work in a big model $\cal U$, over some algebraically closed set $A$
over which $G$  is defined.

{\em Remark}: We know that in the group $G$, because of finite $U$-rank,   there are only a finite number of orthogonality classes of minimal
   types, hence we could work over some model $M_0$ over which   they are
   all represented;  on the other hand,  the arguments  below reprove that  there are only finitely many such classes. \\

\begin{lemma} \label{maximalQminimal} Let $Q$ be a minimal type-definable set,  contained in $G$.  Then there is a
 unique  maximal  connected type-definable subgroup $B_{Q}$ of $G$ which is almost minimal with respect to $Q$. 
Furthermore, if $Q'$ is a minimal type extending $Q$ then
 $Q'/B_Q$ is trivial.
\end{lemma}


\begin{proof}  Let  $Q'$ be the unique  {\em minimal type} extending $Q$ over the same set of parameters as $Q$ (i.e. complete stationary type of $U$-rank $1$, extending $Q$). 

First let $a\in {\cal U} $ realize the type $Q'$, and let $Q_{1}' =  Q'- \{a\} := \{ x \ ;\  x = b-a ,  b\models  Q'\}$, $Q_{1}'$ is an indecomposable type-definable set
   containing
  zero (\ref{stationaryindecomposable}). So by the  indecomposability theorem (\ref{indecomposability}), the group generated by
  $Q_{1}'$, $H(Q_{1}')$,  is type-definable and connected. Clearly 
$Q' \subset H(Q_{1}' )\cup (a+ H(Q_{1}'))$, so  $Q'/H(Q_{1}')$ is finite and by indecomposibility of $Q'$, of cardinality one. 

Now consider the class of all connected type-definable subgroups of $G$, which are almost minimal with respect to $Q$,  denoted by
$S_Q$. By finite $U$-rank of $G$, one can choose such a group, $B_Q$, of
maximal $U$-rank.  Then, for any other  connected type-definable group, $H$
in $S_Q$, $B_Q + H$ is connected and type definable (by the Indecomposability Theorem), and is
also almost minimal with respect to  $Q$. As it contains $B_Q$ it has
bigger or equal rank. By maximality of the rank of $B_Q$, it has same
rank, so by connectedness it is equal to $B_Q$ and it follows that any
such $H$ is contained in $B_Q$. So $B_{Q}$ is the unique type-definable subgroup of $G$ which is maximal connected almost minimal with respect to $Q$. In  particular the group $H(Q_{1}')$ from
above is contained in $B_{Q}$  hence $Q'/B_Q$ has cardinality one. 
\end{proof}

\begin{proposition} \label{existencesocle} There is a type-definable  connected subgroup of $G$,
  $S(G)$,  which is the unique maximal   almost pluri-minimal connected subgroup  of $G$.\\
  Furthermore, $S(G) = B_{Q_1} + \ldots + B_{Q_m}$, for some $m$, where each $Q_i$ is a minimal type-definable set,  and:

  -- $B_{Q_i}$ is the maximal type-definable connected subgroup of $G$ which is almost minimal with respect to $Q_i$

  -- $Q_j \perp Q_i$ if $i \not= j$, so in particular for $i \not= j$, the
  intersection of $B_{Q_i}$ and $B_{Q_J}$ is  finite
  
-- for any other minimal type-definable  $Q$ in
$G$, $Q$ is non orthogonal to one of the $Q_i$.

We call $S(G)$ the (model-theoretic) {\em socle} of $G$.
\end{proposition}

\begin{proof} Consider the class  $S$ of all the connected type-definable subgroups of $G$, $H$, which are  such that, for some finite $F$,  and some $k$, 
$H \subset acl   (F\cup Q_1 \cup \ldots \cup Q_k)$, for $F $ finite and  $Q_i$ minimal
type-definable in $G$. 

The exact same arguments as in Lemma \ref{maximalQminimal}  show that if one
chooses a group  $B$ in
$S$ having maximal U-rank, it will, as above,  be maximal and unique. We denote it $S(G)$. 

\medskip

 Let $R$ be the class of all subgroups (connected and 
  type-definable by the Indecomposability Theorem (\ref{indecomposability}))  of the 
form $H = B_{Q_1} + \ldots  + B_{Q_n}$, where $B_Q$ denotes  the unique maximal subgroup of $G$ almost  minimal with respect to $Q$, given by the previous lemma.  

Choose a group  $H$  in $R$ with maximal U-rank.  As above, $H$
is maximal and unique.

\begin{claim}  $H = S(G)$. \end{claim}

\begin{prclaim}  Of course, by maximality of $S(G)$,  $H \subset S(G)$. 
Suppose for contradiction that  $H$ and $S(G)$ are  not
  equal, then by connectedness of $S(G)$, $S(G)/H$ is infinite of unbounded cardinality. 
Let $M_0$ be a model such that $G, S(G)$ and $B_{Q_1},\ldots, B_{Q_n}$ (and hence also $H$) are defined over $M_0$. As $S(G)/H$ is infinite, if $c$ realizes the generic type of $S(G) $ over $M_0$, then   the coset $ \hat c:= c+H$  is not algebraic over $M_0$. 

By definition of $S(G)$, there are a finite set $F$ (which we can suppose included in $M_0$), and $b_1,\ldots,b_k$, each realizing a minimal type over $M_0$, such that $c \in acl(F, b_1,\ldots ,b_k)$, so $c\in acl(M_0, b_1,\ldots, b_k)$. Choose  $k$
minimal. Let $B = \{b_1, \ldots , b_k\}$ and $B_j := B - \{b_j\}$. By minimality of $k$, for any $j$ , $c\notin acl(M_0 B_j)$.  Now $\hat c$ is algebraic over $M_0 c$, hence over $M_0 B$. Let $B'\subset B$  be maximal  such that $\hat c$ is not algebraic over $M_0 B'$ but is algebraic over $M_0 B' b_j$, for some $b_j \in B\setminus  B'$ (this exists because $\hat c$ is not algebraic over $M_0$).  So  $\hat c$ and $b_j$ are equialgebraic over $M_0 B'$. It follows that $c$ is algebraic over $ M_0 B_j {\hat c}$.  Now  $c$ has the same  $U$-rank as  $b_j$ over $M_0B_j $, that is, it has $U$-rank one. Let $Q$ denote $tp (c/ acl (M_0 B_j))$. Then we have seen earlier that $|Q/B_Q| =1$ and, as by maximality of $H$, $B_Q \subset H$, it follows that $|Q/H| = 1 $ also.  Hence $\hat c$ is in fact algebraic  over $M_0B_j$, and so $c \in acl( M_0 B_j)$, which contradicts the minimality of $k$. 
 \end{prclaim}

\begin{claim} $S(G) = B_{Q_1} + \ldots  + B_{Q_n}$, where the
  minimal  $Q_i$ are pairwise orthogonal  and any other minimal   $Q$ is non
  orthogonal to one of the $Q_i$'s. \end{claim}

\begin{prclaim} If $Q_1$ and $Q_2$ are non orthogonal, then $B_{Q_1} =
B_{Q_2}$:   If they are non orthogonal, they might still be weakly
orthogonal, but,  over some finite set $F$,  $Q_1$
and $Q_2$ are equialgebraic so each $B_{Q_i}$ is almost minimal with respect to the  other.

If $Q_1$ and $Q_2$ are orthogonal, then $B_{Q_1} \cap B_{Q_2}$ is
 contained in the algebraic closure  of a finite set, hence, as $B_{Q_1} \cap B_{Q_2}$ is type-definable, by compactness,  it is finite. 

Suppose some $Q$ were  orthogonal to all the $Q_i$'s. Then $B_Q$  could not be included in $S(G)$, contradicting the maximality. \end{prclaim}

\noindent This finishes the proof of the proposition. 
\end{proof}

Note that this reproves that in $G$ there are only finitely many non
orthogonality classes of minimal types.

If $G$ is a definable group with finite Morley rank, then we can work with definable minimal sets, namely strongly minimal sets, and the groups $B_{Q_i}$ are almost strongly minimal (see for example \cite{Hrushovski} or \cite{Poizatbook}). Recall also that in the finite Morley rank case, every type-definable subgroup of $G$ is in fact definable.\\

So in that case, Proposition \ref{existencesocle} gives the  following:

\begin{proposition}\label{omegastablecase}{\em The finite Morley Rank case: } 
Let $G$ be a definable connected commutative group of finite Morley rank. Then there is a definable connected subgroup of $G$, $S(G)$ which is the unique maximal connected almost pluri-minimal subgroup of $G$. \\
Furthermore, $S(G) = B_{Q_1} + \ldots + B_{Q_m}$, for some $m$, where each $Q_i$ is a strongly minimal set  and:

  -- $B_{Q_j}$ is the maximal definable connected subgroup of $G$ which is almost minimal with respect to $Q_i$

  -- $Q_j \perp Q_i$ if $i \not= j$, so in particular for $i \not= j$, the
intersection of $B_{Q_i}$ and $B_{Q_J}$ is almost finite (contained  in
the  algebraic closure of a finite set)

-- for any other strongly minimal set $Q$ in
$G$, $Q$ is non orthogonal to one of the $Q_i$.
\end{proposition} 

In this section there has been a blanket assumption that the ambient theory $T$ is stable.  But in fact in Proposition 2.11 for example the assumption is not required, noting that finite Morley rank of the definable group $G$ is in the sense of the ambient theory, and implies stable embeddability of $G$.

Finally some remarks about sets of definition for $S(G)$: In the finite Morley rank case,  the finite number of non orthogonality classes of strongly minimal sets in $G$ are represented over any model over which $G$ is defined. In the  finite $U$-rank case, as mentioned above, in order to have representatives for all the finitely many non orthogonality classes of complete minimal types, one might need to work over a slightly saturated model. But, in both cases,  if $G$ is defined over some  set $A$, then by uniqueness and maximality,  each of the type-definable groups $B_{Q_i}$ and the type-definable group $S(G)$ itself will be left invariant by any $A$-automorphism. It follows that they are all in fact defined also over $A$. \\

\subsection{The Weak Socle Theorem}

Here we will state and prove what we call the ``weak socle theorem" for type-definable groups of finite $U$-rank. This was originally proved by Hrushovski in
\cite{Hrushovski}, although,  in the paper, he only  gave the proof for the
finite Morley rank case. There are a few steps where one has to be a 
little more careful when dealing with type-definable groups and
finite $U$-rank.  The reason we call it the weak socle theorem is as in \cite{BBP2}, namely that a stronger statement which is sometimes referred to as the socle theorem (Theorem 2.1 of \cite{Pillay-ML}) was proved in the context of algebraic $D$-groups

We work in  a saturated model $\cal U$ of a stable theory $T= T^{eq}$,
$G$ will be a connected type-definable group of finite $U$-rank.  We
will assume $G$ to be commutative, and will use additive notation, but the theorem works in general.   Let $S(G)$ be the
model-theoretic socle of $G$ as we have defined it in the previous section.  Assume $G$, so also $S(G)$,  are type-defined over $\emptyset $.
We will be making use of the stability-theoretic fact that if $D$ is a type-definable set (over parameters $A$ say) and $a$ is a finite tuple, then there is a small subtuple $c$ from $D$ such that $tp(a/A,c)$ implies $tp(a/D,A)$.  It can be proved by using the fact that $tp(a/D,A)$ is definable (Proposition 2.19 of \cite{Pillay-stability}).


\begin{theorem}\label{weaksocle} Let $p(x)$ be a complete stationary type (over
  $\emptyset$ say) of an element of $G$. Assume $Stab_{G}(p)$ is
  finite. Assume also that every connected type-definable subgroup of
  $S(G)$ is type-defined over $acl(\emptyset)$ (ie $S(G)$ is rigid). Then there is a coset (translate) $C$ of $S(G)$ in $G$ such that all realizations of $p$ are contained in $C$.
\end{theorem} 
\begin{proof}  Without loss work over $acl(\emptyset)$, so $\emptyset = acl(\emptyset)$.  By stability (\cite{Pillaybook}, Lemma 6.18) $S(G)$ is an intersection of relatively definable subgroups $L_{i}$ of $G$, all over $\emptyset$.  As $G$ has finite $U$-rank, there is $L_{i}$ such that $S(G)$ is the connected component of $L_{i}$. 

  Note that if the set $X$ of realizations of $p$ is contained in a single coset of $L_{i}$, then it is also contained in a single coset of $S(G)$: such a coset $C$ of $L_{i}$ will be defined over $\emptyset$ (it must be left invariant by every automorphism), and as $S(G)$ has bounded index in $L_i$ and $p$ is stationary, $X$ will also  be contained in a single coset of $S(G)$.

So now we write $L$ for $L_i$. As  $L$ is  relatively definable in $G$,  $G/L$ is also a type-definable group (in $T^{eq}$).  The group $L$ is not  connected, but  our assumptions imply that every connected type-definable subgroup of  $L$ is still type-defined over $\emptyset$ (as it is contained in the connected component of $L$, $S(G)$).

\smallskip 

\noindent 
{\em Claim 1. }If $X$ is  contained in $acl(L(\U))$, then $X$ is contained in a single  coset of $S(G)$.
\newline
{\em Proof of Claim 1.} If $X \subset acl(L(\U))$, as $X$ is a complete type over $\emptyset$, there is a formula $\phi(x,y)$ , with $y=(y_1,\ldots , y_k)$, which is algebraic for all $y$, such that $p(x) \vdash \exists y_1 \ldots  \exists y_k   ( \land_i  (y_i \in L \land y_i \in Y_i ) \land \phi(x,y))$, and $Y_i$ is a complete type over $\emptyset$. As $S(G)$ is the connected component of $L$,  $L$ is the union of a bounded number of cosets of $S(G)$ hence all elements of the complete type $Y_i$ are in the same coset of $S(G)$. Let $d_i$ be an element of $Y_i$, then $X \subset acl (d_1,\ldots ,d_k, S(G))$, so in particular $X$ is semi-pluriminimal.  As $X$ is a stationary type, it is indecomposable (\ref{stationaryindecomposable}). Let $ a \in X $, then $X-a$ is indecomposable and contains the identity, so  generates a connected infinitely definable group, which is  almost pluri-minimal (as $X$ is), and hence must be included in $S(G) $. So $X-a$ is contained in $S(G)$.  \qed

Let $\pi: G\to G/L$ be the canonical projection.  Let $a$ realize $p$ and let $b =\pi(a)$.  So $tp(b)$ is stationary too. If it is algebraic then by stationarity, it is a single point, whereby $X$ is contained in a single coset of $L$, hence, as noted above, also in a single coset of $S(G)$ as required.  Otherwise $tp(b)$ is nonalgebraic, and we aim for a contradiction.

Let $G_{b} = \pi^{-1}(b)$ and $X_{b} = X\cap G_{b}$.   Note that $G_{b}$ is a coset of $L$ so $L$ acts on $G_{b}$ by addition.  We will show, in Claims  2 and 3 below, together with the assumption that $Stab(p)$ is finite, that in fact $G_{b}\subset  acl(L,b)$  (even in the case when $tp(b)$ {\em is} algebraic).  The argument is in hindsight a Galois-theoretic argument  (finiteness of a certain definable automorphism group).  When $tp(b)$ is {\em not} algebraic, we get our contradiction by showing that the socle of $G$ has to properly contain $S(G)$. 

\smallskip

\noindent
{\em Claim 2.} There is a type-definable subgroup  $H$ of $L$ such that each orbit  of $H$ contained in $G_{b}$ is the set of realizations of a complete type over $Lb := L\cup\{b\}$  and moreover these are precisely the complete types over $Lb$ realized in $G_{b}$.  
\newline
{\em Proof of Claim 2.}  Consider any $a'$ in $X_b$. As remarked just before the statement of the Theorem, there is a small tuple $c$ from $L$ such that $r := tp(a'/c,b)$ implies $tp(a'/L,b)$.  Let $Z$ be the set of realizations of $r$. 

As $Z$ is also a complete type over $Lb$, for each $s\in L$, $s+ Z$ is  the set of realizations of a complete type over $Lb$. Moreover $s+Z \subset G_b$ and for any $x \in G_b$ , there is some $s \in L$ such that $x-a' =s$, so $x \in L+Z$, hence $G_b$ is the union of the $s+ Z$, which are pairwise disjoint or equal, as they are complete types over $Lb$.

  It remains to see that $Z$ (and so each of its $L$-translates) is an $H$-orbit for some type-definable subgroup $H$ of $L$.   Let 
  $H  =\{h\in L: h+a'\models r\}$.  Then $H$ is a type-definable subset of $S$, which  is clearly   independent of the choice of $a'$ realizing $r$ (as any two realizations of $r$ have the same type over $Lb$).  It follows that $H =\{h\in S: h+Z = Z\}$ is the stabilizer of the complete type $r$, and is hence a type-definable group (over $cb$). It follows also that $Z = a'' + H$ for any $a'' \in Z$: as $s+Z$ and $Z$ are either equal or disjoint, for any $a_1,a_2 \in Z $, $a_1-a_2 \in H$.

  Now for $s \in L$, $s+Z = s+a+H$, so is an $H$-orbit. Hence $G_b$ is is a union of $H$-orbits which are each a complete type over $Lb$. \qed

\vspace{2mm}
\noindent
As $X_{b}$ is a type-definable subset of  $G_{b}$ (over $b$),  it is also a union of such $H$-orbits. Let $H^0$ be the connected component of $H$, so  by assumption   $H^0$ is type-definable over $\emptyset$. As $H$ itself is a (bounded) union  of $H^0$-orbits, so is $X_b$. 

\smallskip
\noindent 
{\em Claim 3.} $H^{0}$ is contained in $Stab(p)$. 
\newline
    {\em Proof of Claim 3.} Let $h\in H^{0}$  independent from $a$ over $\emptyset$,  then $h +a \in X_b $ by the previous claim, so $h+a$ also realizes $p$. As $p$ is stationary , then for any $d$ realizing $p$ and independent from $h$ over $\emptyset$, $h+d$ will also realize $p$. \qed

\vspace{2mm}
\noindent
It follows from Claim 3 and  our assumption that $p$ has finite stabilizer,  that $H^{0}$ is finite, so trivial, and so $H$ is finite. Thus also the translate $Z$ of $H$ is finite. 
So $Z $, the set of realizations of $r = tp(a/b,c)$,  is defined by an algebraic  formula $\phi(x,b,c)$.   

Now we make use of the assumption  that $tp(b)$ is nonalgebraic. Note that $tp(b/c)$ remains non algebraic. If not, then $Z \subset acl(c)$, so $a \in acl(L)$, which was ruled out in Claim 1.


Let $Y$ be the set of realizations of $tp(b/c )$. We now work over $c$.  Let $Z'$ be the union of all the sets  defined by formulas $\phi(x,b',c)$ where $b'\in Y$.   Then $Z'$ is a type-definable (over $c$)  subset of $G$ which projects finite-to-one on $Y$. Let $Q$ be a complete minimal type extending $Z'$ over some finite $F \supset c$.  Then $Q/L$ must be infinite as  $Q \cap \pi^{-1} (b')$ is finite for every $b' \in Y$.  Let $d \in Q$, then $Q- d$ is indecomposable and contains the identity, so it generates a connected type-definable almost minimal subgroup in $G$,  which must hence be contained in $S(G)\subset L$, contradicting that $Q/L$ is infinite.  
\end{proof}

\section{The Socle for semiabelian varieties}\label{socleGsharp}

We study the model-theoretic socle for groups of the form $G^{\sharp}$ (described below)  where $G$ is a semiabelian variety, as part of reducing Mordell-Lang for semiabelian varieties to Mordell-Lang for abelian varieties.  In \cite{BBP2} our new  proof of Mordell-Lang for abelian varieties in characteristic $0$ made use of the weak socle theorem in its finite Morley rank form. 
But our new proof in the same paper of Mordell-Lang for abelian varieties over $\Fp(t)^{sep}$ did NOT make use of the weak socle theorem.  In the positive characteristic case, it is only for semiabelian varieties that the weak socle comes into the picture.  We also insist here that we do everything without appealing to the trichotomy theorem (for $U$-rank $1$-types in differentially closed or separably closed fields).  In fact in the appendix we give new proofs that $A^{\sharp}$ is $1$-based for a traceless simple abelian variety, in those cases where we know  how to prove Mordell-Lang without appeal to the trichotomy theorem.

\subsection{Analyzing the socle of $G^\sharp$} 

We will now analyze the model theoretic socle in the context  of the groups used for the model-theoretic proofs of function field Mordell-Lang, in particular the original proof of Hrushovski (\cite{Hrushovski}) and then use this analysis to  deduce Mordell-Lang for all semiabelian varieties from Mordell-Lang  for abelian varieties. This will be done in Section \ref{ML}. 

So we will consider, in characteristic $0$, a differentially closed field $(K,\partial)$, and $k$ its field of constants and in characteristic  $p$, $K$ a separably closed field of non-zero finite degree of imperfection, which we will assume to be $\aleph_1$-saturated, and $k = \Kpinf := \bigcap_{n} K^{p^n} $, the biggest algebraically closed subfield of $K$. We will in both cases refer to $k$ as the {\em field of constants} of $K$. This terminology makes sense also   in the positive characteristic case, as $K$ can  be endowed with a family  of iterative Hasse derivations, and $K^{p^\infty}$ is then the field of common constants  (see for example \cite{Ziegler} or \cite {BBP1}). In the present  paper, we do not directly use the  Hasse derivations framework. At some point, we will however use the formalism of $\lambda$-functions (for a fixed $p$-basis), which give the coordinates of elements of $K$ viewed as a $K^{p^n}$-vector space. The analogue of the Kolchin topology is the $\lambda$-topology, which is not Noetherian but is the limit of the $\lambda_n$-topologies, which are Noetherian. Quantifier elimination in this language implies that definable sets are $\lambda_n$-constructible for some $n$ (see \cite{Hrushovski} and \cite{BouscarenDelon} for details).

We let $G$ be a semiabelian variety over $K$, and the finite $U$-rank or Morley rank  connected  type-definable groups we will consider are the subgroups $G^\sharp$. We recall below briefly the main definitions and facts about $G^\sharp$ which we will be using. For more  details about definitions and properties of $G^\sharp$ we refer the reader to \cite{BBP1}(in particular Section 3 for the basic properties) and \cite{BBP2}.

From now on $G$ is a semiabelian  variety over $K$, we work in a big saturated  model $\U$, $K\preceq \U$, the field of constants of $\U$ will be denoted $C := C(\U)$, and $k = C(K)$. 

Recall that if $A$ is an abelian variety over some field $K_1$ , and $K_0<K_1$, $K_0$ algebraically closed, the $K_1/K_0$-trace of $A$ is a final object in the category of pairs $(B, f)$,  consisting
of an abelian variety B over $K_0$, equipped with a $K$-map of abelian varieties  $ f : B_{K_1} \mapsto A$, where $B_{K_1}$ denotes
the scalar extension $B \times_{K_0} {K_1}$.  We will just need here the property that some abelian variety $A$ has $K_1/K_0$-trace zero,  which we will just
call  {\em having $K_0$-trace zero},  if and only if  $A$ has no non trivial subabelian  variety $A_0$ isogenous to some abelian variety over $K_0$. 

\medskip 

\begin{definition}{\em{The group $G^\sharp$}.}
\newline (i) In characteristic $0$, $G^{\sharp}$ denotes  the ``Kolchin
closure of the torsion", namely the smallest definable (in the sense of
differentially closed fields) subgroup of $G(\U)$ which contains the
torsion subgroup (so note that $G^{\sharp}$ is definable  over $K$).
\newline
(ii) In positive characteristic, $G^{\sharp}$ denotes
$p^{\infty}(G({\U})) =_{def} \bigcap_{n}p^{n}(G(\U))$, a type-definable subgroup over $K$, it is the biggest divisible subgroup of $G(\U)$. 
\end{definition}

\begin{remark} {\rm In  characteristic $0$, the smallest definable
    subgroup containing the torsion subgroup exists by
    $\omega$-stability of the theory  $DCF_{0}$.  In  positive characteristic,
    $G^{\sharp}$ is only type-definable,  it is also the  smallest
    type-definable subgroup of $G(\U)$ which contains the
    prime-to-$p$ torsion of $G$.  Moreover, as noted, $G^\sharp$ is defined over $K$, and $G^{\sharp}(K)$ 
      is  also the biggest  divisible subgroup of $G(K)$.}    \end{remark}

Some basic facts which hold in all characteristics:

\begin{fact} \label{FactGsharpbasic}(i) $G^\sharp $ is also the   smallest Zariski dense type-definable subgroup of $G(\cal U)$.\\
(ii) $G^{\sharp}$ is connected (no relatively definable subgroup of finite index), and of finite $U$-rank in char. $p$, and finite Morley rank in char. $0$. \\
(iii) If $A$ is a simple abelian variety, $A^\sharp$ has no proper
  infinite 
 definable subgroup. It follows that $A^\sharp$ is almost minimal.  If $A$ is any abelian variety, then $A^\sharp = S(A^\sharp)$. \\ 
(iv)  If $G$ is the sum of a finite number of  $G_{i}$ then $G^{\sharp}$ is the sum of the $G_{i}^{\sharp}$.  \\
(v) If $G$ is a semiabelian variety over the field of constants  $k = C(K)$, then $G^\sharp = G(C)$. 
\end{fact}

\begin{proposition}\label{newGsharp} 
(i) In characteristic $p$, every connected type-definable subgroup  of $G^\sharp $ is of the form $H^\sharp $, for some  closed connected subgroup (i.e. a semiabelian subvariety) $H$ of $G$. It follows that $G^\sharp $ is rigid.\\
(ii) In characteristic $0$, if $A$ is an abelian variety, every connected definable subgroup of $A^\sharp$ is of the form $B^\sharp$ for some abelian subvariety $B$ of $A$. It follows that $A^\sharp$ is rigid.
\end{proposition}
\begin{proof} (i) Let $R$ be a connected type-definable subgroup of $G^\sharp$.   By $U$-rank inequalities, and because for each $n$, the $n$-torsion is finite, $R$ is divisible (see for example \cite{BD}).  Let $H := \overline{R}$,  the Zariski closure of $R$ in $G$. Then as $H^\sharp $ is the biggest divisible subgroup of $H$, $R \subset H^\sharp$. But (Fact \ref{FactGsharpbasic} (i)) $H^\sharp$ is also the smallest type-definable subgroup of $H$ which is Zariski dense in $H$. Hence $R = H^\sharp$. Now every  connected closed subgroup of $G$ is defined over $K$ (by rigidity of semiabelian varieties), so $H$ is defined over $K$ and hence so is  $R= H^\sharp$, and it follows that $G^\sharp $ is rigid.\\
(ii) Let $R$ be a connected definable subgroup of $A^\sharp$. Then $B := \overline{R}$, the Zariski closure of $R$ in $A$, is an abelian subvariety of $A$.  By exactness of the $\sharp$-functor for abelian varieties in characteristic $0$ (Proposition 5.23 in \cite{BBP1}), $B^\sharp = B(\U) \cap A^\sharp$, and so $B^\sharp \supset R$. But (Fact \ref{FactGsharpbasic} (i)) $B^\sharp$ is the smallest Zariski-dense definable subgroup of $B(\U)$, hence $R = B^\sharp$. As in (i) above, it follows that $A^\sharp$ is rigid.\end{proof}

\medskip

\begin{remark}\label{notexact} {\rm In characteristic $0$, it is not true that for $G$, any  semiabelian variety  over $K$, every connected definable subgroup of $G^\sharp$ is of the form $H^\sharp$.  We have shown in \cite{BBP1}(Section 3 and Cor. 5.22)   that there are cases when,  for $G$ of the form: $0 \rightarrow T  \rightarrow G \rightarrow A\rightarrow 0$ , the induced sequence 
    $T^\sharp  \rightarrow G^\sharp  \rightarrow A^\sharp $ is not exact, which is equivalent to the fact that  $R := G^\sharp \cap T(\U)$, which is a connected definable group, strictly contains ${T}^\sharp$. But if $R = H^\sharp$, for some $H$, then $\overline{R} = H < T $ but as ${T}^\sharp \subset R$, $\overline R = T = H$. Contradiction. 
Note that $R/T^{\sharp}$ is a finite-dimensional vector space over the contants, so when it has Morley rank $\geq 2$, $G^{\sharp}$ will not be rigid. 
The same example also gives, in characteristic $p$, an induced $\sharp$-sequence which is not exact, hence such that   $R \not= T^\sharp$, but, in contrast to the characteristic $0$ case, $R$ will not be connected (in fact 
$T^\sharp$ will be the connected component of $R$).}\end{remark} 

\medskip 

Now as the group $G^\sharp$ has finite $U$-rank in characteristic $p$, and finite Morley rank in characteristic $0$, we know by the preceding section, that it has a (model-theoretic) socle,  the maximal almost pluri-minimal subgroup $S(G^\sharp)$. We know that there are a finite number $Q_0,\ldots, Q_n$ of pairwise orthogonal minimal types in the characteristic $p$ case, and of pairwise orthogonal  strongly minimal sets in the characteristic $0$ case, such that 
$S(G^\sharp) = B_{Q_0} +\ldots  + B_{Q_n}$ where each $B_{Q_j}\subset acl ( F_j Q_j)$ is the maximal connected type-definable subgroup of $G^\sharp$ which is almost minimal with respect to $Q_j$. Furthermore, any other minimal (or strongly minimal) $Q$ must be non orthogonal to one of the $Q_j$.

\smallskip 

\begin{remark}{\rm  Whether in characteristic $0$ or $p$,  there is at most one minimal type (resp. strongly minimal set) amongst the $Q_j$ which is non orthogonal
to the field  of constants in $\U$,  $C$: recall that, in the characteristic $p$ case, $C= U^{p^\infty}$ is a pure   type-definable algebraically closed field, hence with a unique  minimal generic type, and that in characteristic $0$, $C$  is a pure definable algebraically closed field, hence a strongly minimal set. This particular minimal type (or strongly minimal se)  will be denoted, in what follows, by $Q_0$. }\end{remark}

\begin{proposition} \label{propertiessocle} In all characteristics: \\
{\bf 1.} If $G$ is a semiabelian variety over $K$ and not an abelian variety, or if $G$ is   an abelian 
variety over $K$, with $k$-trace non zero, then, there is in  $G^\sharp$  a minimal type  (or a strongly minimal set in characteristic $0$), $Q$,  non orthogonal to the field of constants $C$.\\
{\bf 2.}  Let $R$ be any  definable connected subgroup of $G(\U)$  of finite Morley rank in characteristic $0$, or  any connected type-definable subgroup of $G(\U)$ of finite $U$-rank in characteristic $p$. Let $Q$ be  a strongly minimal set (or a minimal type-definable set) in $R$,  non orthogonal to $C$, and let  $R_Q\subset R$  be  any connected  type-definable subgroup of $R$ which is   almost minimal with respect to $Q$. Then there exists a semiabelian variety $H<G$,
$R_Q\subset H(\U)$, $H^\sharp = R_Q$, and  a semiabelian  variety $J$ over $k$ with  an isogeny $f$ from $J$ onto $H$  such that  
$ f^{-1}({H}^\sharp ) = J(C)$.\\
{\bf 3.}  Let $R$ be any  definable connected subgroup of $G(\U)$  of finite Morley rank in characteristic $0$, or  any connected type-definable subgroup of $G(\U)$ of finite $U$-rank in characteristic $p$. Let  $Q$ be  a  strongly minimal set (or a minimal type)  in $R$   orthogonal to $C$ and let   $R_Q\subset R$  be  any connected  type-definable subgroup of $R$,  almost minimal with respect to $Q$. Then  $A := \overline {R_Q}$ is an abelian variety with $k$-trace zero and $R_Q = A^\sharp$. 
 \end{proposition}

\begin{proof}\\
\noindent 1:  If $G$ is a semiabelian variety over $K$, and not an abelian
variety,  $G$ has a non trivial closed connected
subgroup $T$ which is  a torus, that is, isomorphic to a product of
${\mathbb G}_m$ (defined over $C$). Then $T^\sharp $  is isomorphic to 
$({\mathbb   G}_m)^n (C)$. It is hence almost minimal with respect to the strongly minimal field $C$ in characteristic $0$, or  to the minimal generic type of the field $C$ in characteristic $p$.  There must then be in $T^\sharp$, hence in $G^\sharp$,  a  strongly minimal set or a minimal type, $Q$,  which is not orthogonal to the field $C$.\\
If $G$ is an abelian variety with $k$-trace non zero, then there is some $A <G$, $A$ isogenous to some abelian
variety $B$ over $k = C(K)$, it follows that $A^\sharp$ is definably isogenous  to $B(C)$ (by \cite{BBP1} Lemmas 3.3 and 3.4). Hence,   $A^\sharp = S(A^\sharp)$, must be non orthogonal to the field $C$. \\

\noindent 2: First we apply Proposition \ref{nonorthogonal}, so there is a type-definable group $N \subset dcl(C)$ and a definable surjective homomorphism $h : R_Q \mapsto N$, with finite kernel. As in \cite{Hrushovski} (see also \cite{bouscaren} or \cite{BD}),   it follows from the ``pureness'' of the algebraically closed field of constants and from the model theoretic version of Weil's theorem, that $N$ is definably isomorphic to $J(C)$ for some commutative connected algebraic group $J$ over the  field $C$. As $R_Q$ is a divisible group and the $n$-torsion is finite for each $n$, there is a definable ``dual'' isogeny, $f$ from $J(C)$ onto $R_Q$. By quantifier elimination in the algebraically closed  field $C$,  the definable map $f$ is given locally by  rational functions (in characteristic $p$, a priori $f$ is given locally by $p$-rational functions, that is might include some negative powers of the Frobenius, but by changing $J$ one gets also to a locally rational map). \\
      Taking Zariski closures, extend $f$ to a surjective morphism $\tilde{f}$ of algebraic groups from $J$ onto $H:= \overline{R_Q}$. By Chevalley's theorem, there is a connected  unipotent algebraic subgroup of $J$ over $C$, $N$, such that $J/N$ is the maximal semiabelian quotient of $J$. A priori, ${\tilde{f}}$ is no longer an isogeny, but by  maximality, as $H$ itself is a semiabelian variety, $N < Ker {\tilde{f}}$ and as $f$ was an isogeny in $J(C)$, $N(C)$ must be finite, hence trivial,  as $N$ is connected. It follows that $J$ itself is a semiabelian variety. But then  the connected component of $  Ker {\tilde{f}}$ which is a connected closed subgroup of $J$ is defined over $C$, so  must be trivial, hence ${\tilde{f}}$ is also an isogeny.\\
      Now as $J$ is defined over $C$, $J^\sharp = J(C)$ and  as $f$ is dominant, $f(J^\sharp) = H^\sharp $ (\cite{BBP1} Lemma 3.4) and so $R_Q = H^\sharp$. \\
      
\noindent 3: If $\overline{R_Q} = H$, then $H^\sharp \subset R_Q \subset H(\U)$ (as $H^\sharp$ is the smallest definable subgroup Zariski dense in $H$). But then it is also the case that $H^\sharp$ is almost minimal with respect to $Q$ and  must be orthogonal to the field $C$, hence by (i) $H$ must be an abelian variety of $k$-trace zero. Then we know by  Proposition \ref{newGsharp}, that $R_Q$ is of the form $B^\sharp$ for some $B < H$ and we must have $B=H$.    
\end{proof}

\begin{corollary} \label{char.0} 
(i) In characteristic $0$: (a) If $R$ is a connected definable subgroup of $S(G^\sharp)$, then $R = H^\sharp$ for some semiabelian variety 
$H<G$. \\
(b) If $L$ is a connected definable subgroup of $G(\U)$ of finite Morley rank, and $L \supset G^\sharp$, then 
$S(L) = S(G^\sharp)$. \\
(ii) In all characteristics: $S(G^\sharp)$ and all its connected type-definable subgroups are  defined over $K$ and rigid.
\end{corollary}

Before beginning  the proof, let us make immediately some remarks: In characteristic $p$ (a) is true more generally   for all connected type-definable subgroups of $G^\sharp$ itself (Fact \ref{newGsharp}), but we saw that in characteristic $0$ this may be false if $G$ is not an abelian variety (Remark \ref{notexact}).   Secondly (b) is vacuous in characteristic $p$ as any  connected
type-definable subgroup of $G(\U)$ with finite $U$-rank must be divisible, hence contained in $G^\sharp$ which is the biggest such. \\

\noindent\begin{proof} (i)(a) If $R \subset S(G^\sharp)$, then $R$ is almost pluri-minimal, and hence $R = S(R)$. It follows that $R = R_{Q_0} + \ldots +R_{Q_n}$, where the $Q_j$ are pairwise orthogonal strongly minimal sets and $R_{Q_j}$ is a connected definable group, almost minimal with respect to $Q_j$. It follows from  Proposition \ref{propertiessocle} that, for each $j$, $R_{Q_j} = {H_j}^\sharp $ for some $H_j < G$ and so  $R = {H_0}^\sharp + \ldots +{H_n}^\sharp = H^\sharp $ for $H = H_0+ \ldots +H_n$, using the definition of the $\sharp$-group in characteristic zero as the definable closure of the torsion. \\
(i)(b)  Recall that $G^\sharp$ is the smallest definable subgroup of $G$ which contains all the torsion of $G$. It follows that for any $H < G$, $H^\sharp $ , the smallest definable group containing all the torsion of $H$, must be contained in $G^\sharp$. 
Now by Proposition \ref{propertiessocle}, $S(L) =  {H_0}^\sharp + \ldots +{H_n}^\sharp $ 
where $H_j <G$ , hence $S(L) \subset G^\sharp$, By maximality of $S(G^\sharp)$, $S(L) \subset S(G^\sharp)$ and as $G^\sharp \subset L$, $S(G^\sharp ) \subset S(L)$.\\
(ii) follows from the rigidity of $G$, from (i)(a) and from Proposition \ref{newGsharp}.
\end{proof}

\medskip

We can now completely describe $S(G^\sharp)$: 

\begin{proposition} \label{decomposition}If $G$ is semiabelian over $K$, 
we have that $S(G^\sharp) = H^\sharp = {H_0}^\sharp + {H_1}^\sharp + \ldots + {H_n}^\sharp $, where: 

-- $H = H_0 + \ldots + H_n$, where for $i \not= j$, $H_i$ and $H_j$ have finite intersection,

 -- for each $j$, $H_j^\sharp$ is  almost strongly  minimal in characteristic $0$ and almost minimal in characteristic $p$

-- $H_0< H$ is the unique maximal  semiabelian subvariety of $G$ which is isogenous   to some semiabelian variety 
 defined over $k$. Note that $H_0$ will be trivial iff $G$ is an
abelian variety with  $k$-trace zero

-- for $j > 0$, $H_j<H$ is an abelian variety with $k$-trace zero  

-- for $i \not= j$, ${H_j}^\sharp $ and  ${H_i}^\sharp$ are orthogonal and have finite intersection 

-- $A_G := H_1 +\ldots +H_n$ is the unique maximal  abelian subvariety of $G$ with $k$-trace zero.

\end{proposition}

\begin{proof}
Following the notation from the previous section, we know that $S(G^\sharp) = B_{Q_0} + B_{Q_1}+ \ldots + B_{Q_n}$, where the $Q_j$ are minimal type-definable sets  (definable in the char. $0$ case), and where we decide that $Q_0$ denotes the (generic type of) the field of constants. Each $B_{Q_j}$ is the maximal type-definable connected subgroup of $G^\sharp$ which is almost minimal with respect to $Q_j$. The $Q_j$ are pairwise orthogonal and any other minimal $Q$  in $G^\sharp $ is non orthogonal to one of the $Q_j$.

By  Proposition \ref{propertiessocle},  $B_{Q_0} =0$ if and only if $G $ is an abelian variety with $k$-trace zero. 

By Proposition \ref{newGsharp} and Corollary \ref{char.0}, $S(G^\sharp) = H^\sharp$ for some $H <G$, and for each $j$, $B_{Q_j} = {H_j}^\sharp $ for some $H_j< H$. It follows that $H = H_0 + H_1+\ldots + H_n$. \\
By Proposition \ref{propertiessocle},  $H_0 = \overline{B_{Q_0}} $ is isogenous to some $D$ defined over $k$. We claim that it is maximal such in $G$, that is that any other $J < G$, isogenous to some semiabelian variety $E$ over $k$ must be contained in $H_0$: indeed then $J^\sharp$ is definably isogenous to  $E^\sharp = E(C)$.  It follows that $E^\sharp $ is almost minimal with respect to $Q_0$, and as $J^\sharp =  f^{-1}(E(C))$ and $f$ has finite kernel, $J^\sharp$ is also almost minimal with respect to $Q_0$. But then,  by maximality of $B_{Q_0}$,  we have that  $J^\sharp \subset B_{Q_0}$ and passing to Zariski closures, $J < H_0$. 

As for $j >0$, $Q_j$ is orthogonal to $Q_0$ hence to the field of constants, by Proposition \ref{propertiessocle}, each $H_j := \overline{B_{Q_j}} $ must be an abelian variety of $k$-trace zero. 

Let $A_G:= H_1+\ldots +H_n$, $A_G$ is an abelian variety with $k$-trace zero. We claim  that it is maximal such in $G$: Suppose there is another one, $R$. Then as $S(R^\sharp) = R ^\sharp$ , $R^\sharp = R_1 + \ldots +R_m$, where each $R_i$ is almost minimal with respect to a (strongly) minimal $P_i$. For each $i$, $P_i$ must be non orthogonal to one  of the minimal $Q_j$. It cannot be $Q_o$ because $R$  has $k$-trace zero, hence for some $j >1$, by maximality of $B_{Q_j}$, $R_i \subset B_{Q_j}= {H_j}^\sharp$, and $R ^\sharp \subset {A_G}^\sharp$. Passing to Zariski closures, $R < A_G$. 

Finally, we know that for $i \not= j$, $Q_j$ and $Q_i$ are orthogonal. It follows that, in characteristic $zero$, the two almost strongly minimal groups $B_{Q_j}$ and ${B_{Q_i}}$ have finite intersection. In characteristic $p$ we can a priori conclude only that the two almost minimal groups  $B_{Q_j}$ and ${B_{Q_i}}$ have an intersection which is contained in the algebraic closure of a finite set. But as this intersection is also type-definable, it will have to be finite (by compactness). 
\end{proof}

We will give some examples of socles of $G^\sharp $ in the  case of semiabelian varieties in \ref{calculations},  

But for now we summarize what we will need in the next section:

\begin{corollary}\label{socledescribe}
If $G$ is semiabelian over $K$, $S(G^\sharp) = H^\sharp  = G_0^\sharp + A_G^\sharp$, where $G_0 <G $ is the unique maximal  semiabelian  variety  isogenous to one over the constants $k$, and $A_G$ is the unique maximal  abelian subvariety $A_G <  G$ with $k$-trace zero and $H = H_0 + A_G$. Furthermore, $G_0$ and $A_G$ have finite intersection and $G_0^\sharp$ and $A_G^\sharp $ are orthogonal. 
\end{corollary}

\subsection{The algebraic socle}\label{algebraicsocle}

We introduce here a possibly new algebraic-geometric notion,  building on and motivated by our analysis of the model-theoretic socle of $G^{\sharp}$ in the previous section.  So given a pair $K_{0} < K_{1}$ of algebraically closed fields, and a semiabelian variety $G$ over $K_{1}$, we will define below the {\em $K_{0}$-algebraic socle of $G$}, a semiabelian subvariety of $G$.  In Section 4 we will show how (function field) Mordell-Lang for $G$ reduces to (function field) Mordell-Lang for its $K_{0}$-algebraic socle, and then use this to reduce further to (function field) Mordell-Lang for abelian varieties. Both reductions use  model theory and the weak socle theorem. 

\medskip

\begin{definition} Let $K_0<K_1$ be  two algebraically closed fields in any characteristic and $G$ be a semiabelian variety over $K_1$. We   define the {\em algebraic $K_0$-socle} of $G$ as follows:
$$ S_{K_0}(G) := G_0+ A_G$$
where $G_0$ is the maximal  closed connected algebraic subgroup of $G$ which is isogenous to  $H \times_{K_0} K_1$, for some $H$ defined 
over $K_0$ (note that in characteristic zero, one can replace isogenous by isomorphic, but not in characteristic $p$); $A_G$ is the maximal  abelian subvariety of $G$ with $K_0$-trace zero. \end{definition}

The following lemma should be clear but we give the proof for completeness. 

\begin{lemma} \label{descending}
  (i) $G_0$ and $A_G$ are well defined, and if $G$ is an abelian variety, then $S_{K_0} (G) = G$.\\
  (ii)  $G_0$ and $A_G$ have finite intersection.\\
  (iii) If $G = G_1+G_2$, then $S_{K_0} (G) = S_{K_0} (G_1)+S_{K_0} (G_2)$.\\
(iv)  If $G$ is defined over $K_1$, $K_0 < K_1< L_1$, $K_0 < L_0 < L_1$ algebraically closed fields such that $L_0$ and $K_1$ are linearly disjoint over $K_0$, then  $(S_{K_0} (G))\times_{K_1} L_1 = S_{L_0}(G\times_{K_1} L_1)$.\\
(v) If  $G$ is defined over some $K$, $K_0 < K$, $K$ separably closed, then $G_0$, $A_G$ and $S_{K_0}(G)$ are  also defined over  $K$. \end{lemma} 

\begin{proof}
(i) It suffices to check that if $H_1, H_2 <G $ are closed connected subgroups isogenous to groups defined over $K_0$, then so is $H_1 + H_2$, which is itself closed and connected,  and that if $A_1, A_2$ are two abelian varieties in $G$ with $K_0$-trace zero, then the abelian variety $A_1+A_2$ also has $K_0$-trace zero.

If $G$ is itself an abelian variety, then $G$ is a finite sum of simple abelian varieties, $G_i$. Each $G_i$ is either isogenous to some $H$ defined over $K_0$, or has $K_0$-trace zero.

(ii) is clear: if $G_0 \cap A_G$ was infinite, then its connected  component would give $A_G $ a non zero $K_0$-trace.

(iii) Let $S_{K_0} (G_1) = H_1 + A_1$ and $S_{K_0} (G_2)= H_2 + A_2$. Then it is easy  to check that $H_1+H_2$ is the maximal connected subgroup of $G_1+G_2$ which is isogenous to some $H$ over $K_0$ and that $A_1+A_2$ is the maximal  abelian variety in
$G_1+G_2$ with $K_0$-trace zero. 

(iv) It suffices to show that if $H <G\times_{K_1}  L_1$ and there is some $ D $ over $L_0$ and an isogeny from $D  \times_{L_0} L_1$ onto $H$, then there is some $D_0$ over $K_0 $  and an isogeny  $f$  from $D_0 \times_{K_0} L_1$ onto $H$. Note that $H$ itself is in fact defined  over $K_1$ (a connected closed subgroup of $G$), hence it is of the form $H'\times_{K_1} L_1$, for $H' $ over $K_1$. 

There is a rather direct proof  using model theory:  the linear disjointness condition means exactly that, in the language of pairs of algebraically closed fields, the pair $(K_0 , K_1)$ is an elementary  substructure of  the pair $(L_0,L_1)$ (see \cite{Robinson} for model-theory of pairs of algebraically closed fields). One can then argue as in \cite{Hrushovski} or \cite{bouscaren}, expressing the existence of an isogeny between $H$ and a commutative algebraic group over the small model of the pair by a first order formula in the big pair, formula  which by elementary substructure will also be true in the small pair.

For purely algebraic proofs, in the  case of characteristic $0$, one can suppose that the isogeny is in fact  an isomorphism and the conclusion follows classically using linear disjointness.

In characteristic $p$ one cannot replace $f$ by an isomorphism. For the case of abelian varieties, the result  follows for example from more general results about the behaviour of the Trace under field extensions, shown for example in \cite{conrad}(Theorem 6.8).   

(v) This follows directly from the fact that if K is separably closed and $G$ is defined over $K$, then every connected closed subgroup of $G$ is also defined over $K$. 
\end{proof}


\bigskip 

It follows from what  we saw in section \ref{socleGsharp}  that, if $K$ is differentially closed or separably closed (sufficiently saturated), if $k$ denotes the field of constants of $K$, and if $G$ is semiabelian over $K$, then: 

\begin{proposition} \label{lineardisjointness}1. $S(G^\sharp) = (S_{k}(G))^\sharp = {G_0}^\sharp + {A_G}^\sharp$. \\
Passing to Zariski closures, $\overline{S(G^\sharp) } = S_{k}(G)$.\\
2. If $K'\preceq K$ are two differentially closed fields in characteristic $0$ ( $\preceq$ here denotes elementary  extension in the language of differential fields)   or separably closed fields of finite imperfection degree in characteristic $p$ ($\preceq$ then refers to  a suitable language for separably closed fields, considered either as Hasse differential fields or fields with $\lambda$-functions),   if $k',k$ are their respective fields of constants,  and if $G$ is defined over $K'$, 
then   $S_{k}(G\times_{K'} K) = S_{k'} (G)\times_{K'} K$.
\end{proposition}
\begin{proof}  1. is clear. \\
 2. follows from  \ref{descending} and the linear disjointness of $k$ and $K'$ over $k'$. If $S_k(G\times_{K'} K) = G_0 + A_G$, as above, note that both $G_0$ and $A_G$ are defined over $K'$, so of the form $G_0 = (G'_0)\times_{K'} K$ and  $A_G= A'\times_{K'} K$. Then $A_G$  is also the maximal  abelian  variety in $G\times_{K'} K$ with $k'$-trace $0$, $A'$ is the maximal abelian variety in $G$ with $k'$-trace zero,  $G_0$ is also the maximal semiabelian variety in  $G\times_{K'} K$   isogenous to one  defined over $k'$ and $G'_0$ is the maximal semiabelian variety in $G$ isogenous to one over $k'$.   \end{proof}

\section{Mordell-Lang}\label{ML}

We first recall the general  statement of function field Mordell-Lang for semiabelian varieties (for more on the conjecture and equivalent statements see for example \cite{bouscaren} and \cite{Hindry2}):

\medskip

\noindent {\bf Statement of Function field Mordell Lang for semiabelian varieties, all characteristics}: Let $K_0 <K_1$ be two algebraically closed fields, 
$G$ be a semiabelian variety over $K_1$, $X$ an irreducible subvariety of $G$  over $K_1$, and $\Gamma \subset G(K_1)$, a finite rank subgroup. Suppose that $\Gamma \cap X$ is Zariski dense in $X$ and that the stabilizer of $X$ in $G$ is finite. Then there is a semiabelian subvariety $H_0$ of $G$,  a semiabelian variety, $S_0$ over $K_0$,  an irreducible subvariety $X_0$ of $S_0$ defined over $K_0$ and a bijective morphism $h$ from $S_0$ onto $H_0$ such that $h(X_0) = a+X$ for some $a\in G(K_1)$. 

\smallskip 

\noindent{\em Remark}: Note that equivalently, the conclusion above can be stated as 
there is a semiabelian subvariety $H_0$ of $G$,  a semiabelian variety  $S_0$ over $K_0$,  an irreducible subvariety $X_0$ of $S_0$ defined over $K_0$ and an isogeny $f$ from $H_0$ onto $S_0$ such that $X= a+f^{-1} (X_0) $  for some $a\in G(K_1)$.\\

\noindent Note also that if $G$ is an abelian variety over $K_1$ with $K_0$-trace zero, then the conclusion is that $X = \{a\}$ for some $a \in G(K_1)$.

\smallskip 
\noindent Recall that we say that a subgroup $\Gamma$  of $G(K_1)$ has ``finite rank'' if:
in characteristic zero, $\Gamma \otimes  \mathbb Q $ has finite dimension (equivalently $\Gamma$ is contained in the divisible hull of some finitely generated group $\Gamma_0\subset G(K_1)$), and  in characteristic $p$, $\Gamma \otimes  {\mathbb Z}_{(p)} $
has finite rank as a module  (equivalently  $\Gamma$ is contained in the $p'$-divisible hull of some finitely generated group $\Gamma_0 \subset G(K_1)$).  Note that, in characteristic $p$,  this is a restrictive  notion of ``finite rank''. It is still open whether the same statement holds, in characteristic $p$,  with the more general assumption that
$\Gamma \otimes  {\mathbb Q}$ has finite dimension.

\subsection{Recalling the model-theoretic setting}

We recall the reduction to the differential model-theoretical setting, as in Hrushovski's original proof in \cite{Hrushovski} (see also \cite{bouscaren} or \cite{BBP2}):

-- In characteristic $0$, we can replace the pair $K_0<K_1$ by a pair $k<K$ where $(K,\delta)$ is a differentially closed field (a model of the theory $DCF_0$) and $k$ is the field of constants of $K$, $K_0 <k$,  $K_1 < K$ and $K_1$ and $k$ are linearly disjoint over $K_0$. Then there exists a definable (in the sense of $DCF_0$) connected group $H$, with finite Morley rank, such that $H$ contains both $\Gamma$ and $G^\sharp$. It follows that $X \cap H$ is Zariski dense in $X$. By Proposition \ref{char.0}, the (model theoretic) socle of the group $H$ is equal to the socle of $G^\sharp$. 

-- In characteristic $p$, we can replace the pair $K_0<K_1$ by a pair $k<K$, where $K $ is an $\omega_1$-saturated separably  closed field of finite degree of imperfection (a model of the theory $SCF_{p,e}$, with $e$ finite and non-zero), and $k = \Kpinf := \bigcap_{n} K^{p^n} $ (we also call $k$, which is the biggest algebraically closed subfield of $K$,  the field of constants of $K$).  We have that $K_0 < k$  and $K_1$ and $k$ are linearly disjoint over $K_0$, $G$ is defined over $K$, and $\Gamma \subset G(K)$. Then we can replace $\Gamma$ by the connected type-definable group $G^\sharp = p^\infty G(K) := \bigcap_{n} p^n G(K)$, which has finite $U$-rank, and is such that for some $a\in G(K)$, $(a+X) \cap G^\sharp$  is Zariski dense in $a+X$.  Without loss of generality, by translating, we can suppose that  $X \cap G^\sharp$  is Zariski dense in $H$ and, by changing $\Gamma$ to another finite rank group, that it is still true that $X\cap\Gamma$ is Zariski dense  in $X$.

\begin{claim} \label{densityoftype}
(i) In characteristic $0$: There is a complete stationary type $q_X$ in $X \cap H$ such that $ q_X(K)$ is Zariski dense in
  $X$ and $Stab (q_X)$ in $H$ is finite.  \\
(ii) In characteristic $p$: There is a complete stationary type $q_X$ in $X \cap G^\sharp$  such that $ q_X(K)$ is Zariski dense in
  $X$ and $Stab (q_X)$ in $G^\sharp$ is finite.  \end{claim}

\begin{proof}

\noindent In  characteristic $0$: Consider the set $Z := X \cap H$ which is closed in the $\delta$-topology in $G(K)$. By Noetherianity of the topology, $Z$ is a union of finitely many irreducible components, so one of them,  $Z_1$ is also Zariski dense in $X$. Let $q_X$ be the complete $DCF_0$-type which is the topological generic of $Z_1$. Then $Y := q_X(K)$ is also Zariski dense in $X$: indeed $\overline Y$, the Zariski closure
 of $Y$,  contains the $\delta$-closure of $Y$, which is equal to $Z_1$. So
 $\overline{Y} = \overline{Z_1} = X$. \\
Now suppose $g\in G$ stabilizes the type $q_X$, that is,  $(g+q_X(K)) \cap q_X(K)$ contains ${q_X}_{|g} (K)$. But the $\delta$-closure of ${q_X}_{|g} (K)$ is still
equal to $Z_1$, so $(g+q_X(K)) \cap q_X(K)$  is also Zariski dense in $X$  and it follows that 
$g+X= X$. So the stabilizer of the type $q_X$ must be finite. \\
\noindent In  characteristic $p$: $X \cap G^\sharp = \bigcap Y_i$, where the $Y_i$ are a
  decreasing sequence of $\lambda_i $-closed sets.  Each of the
  $Y_i$ is closed in the $ \lambda_i$-topology, which is
  Noetherian, so has a finite number of irreducible components  
for this topology. Every component of $Y_{i+1} $ is contained in a
  component of $Y_i$. So we obtain a tree of $\lambda$-closed irreducible
  sets of finite type, branching finitely. On a given branch we obtain an irreducible
  component of $X \cap G^\sharp$.
As the tree branches finitely, at each level one of the  closed sets
must be Zariski dense in $X$. Taking the intersection, by $\omega_1$-saturation of $K$, we obtain an
irreducible component $Z_1$  of  $ X\cap G^\sharp$ which is Zariski dense in $X$. Then, as in the characteristic $0$ case,  we take $q_X$ to be the complete $SCF_{p,1}$-type which is the topological generic of $Z_1$.  Then
$ q_X(K)$ is also Zariski dense in $X$. The same arguments as above show that the stabilizer of $q_X$ must also be finite.  \end{proof}

\subsection{Reduction to the algebraic socle}

We have, in the previous section, defined what we called, for any pair of algebraically closed fields, $K_0<K_1$, and any semiabelian $G$ over $K_1$, the {\em $K_0$-algebraic  socle of $G$}. 

\medskip 

\noindent The following reduction follows from  the ``weak socle theorem'':

\begin{proposition}\label{reducingtoalgebraicsocle} In order to prove function field Mordell-Lang for all semiabelian  varieties, it suffices to prove it for semiabelian varieties of the form $G_0+A$, where $G_0 $ is a semiabelian variety isogenous to one over $K_0$ and $A$ is an abelian  variety of $K_0$-trace zero.\end{proposition}

Indeed, going back to the ``weak socle theorem'' (\ref{weaksocle}) and its application to the algebraic situation (section \ref{algebraicsocle}), we see that,
 passing to Zariski closures, what we have proved at the  algebraic level is the following:

\begin{proposition}\label{algebraicweaksocletheorem}  Let $K_0<K_1$ be any pair of algebraically closed fields, $G$ a semiabelian variety over $K_1$,  $X$   an irreducible subvariety of $G$ over $K_1$  such that the stabilizer of $X$ in $G$ is finite. Let $\Gamma$ be a subgroup of $G(K_1)$ of finite rank, if  $X\cap \Gamma$ is dense in $X$, then for some $a\in G(K_1)$ $a+X \subset S_{K_0}(G)$.\end{proposition}
\begin{proof}
We pass to the model theoretic setting described in the previous section.
 
In both characteristics, we apply the weak socle theorem (Theorem \ref{weaksocle}) to the type $q_X$ obtained from Claim \ref {densityoftype} and   to the group $H$ in characteristic zero and  to the group $G^\sharp $ in characteristic $p$.  As we know that  in characteristic zero, the socle of $H$ is equal to the socle of $G^\sharp$ (\ref{char.0}), in both cases we  conclude that a translate of the type $q_X$ is contained in the socle of $G^\sharp$, $S(G^\sharp)$.  Taking Zariski closures, it follows that a translate of $X$ is contained in the algebraic $k$-socle of 
$G$, $S_k (G)$. Now as $G$  and $X$ were  defined over the algebraically  closed field $K_1$, by Lemmas \ref{descending} and \ref{lineardisjointness}, a translate of $X$ will be contained in $S_{K_0} G$. \end{proof}

\begin{remark} {\em The condition that $\Gamma$ is a group of finite rank is used  to pass to a group of finite rank in the  sense of model theory, in order to use the model theoretic weak socle theorem.  One might ask whether a purely algebraic-geometric statement suffices, for example whether if $X$ is an irreducible subvariety of $G$ with finite stabilizer, then  $X$ is contained in a translate of $S_{K_{0}}(G)$. But it is not the case, as the following simple example shows:  Let $A$ be a simple abelian variety over $K_{1}$, of dimension $>1$ and with $K_{0}$-trace $0$.  Let $G$ be a non almost split extension of $A$ by $\Gm$ (see Section \ref{calculations}).  Then $S_{K_{0}}(G) = \Gm$.  If $X$ is a curve in $A$, and $Y$ is a curve in $G$ projecting onto $X$, then $Stab_{G}(Y)$ is finite, but $Y$ is not contained in a translate of $\Gm$.  }\end{remark}

\noindent
We can now prove Proposition \ref{reducingtoalgebraicsocle}:\\
\begin{proof} Suppose function field Mordell-Lang  is true for ``socle-like'' semiabelian varieties (i.e. as in the hypothesis of Proposition 4.2),  $K_0<K_1 $ are algebraically closed fields, and  $G$ is  any semiabelian variety over $K_1$. Let $\Gamma$ be a finite rank subgroup of $G(K_1)$, and $X$ an irreducible subvariety of $G$ over $K_1$, such that $X\cap \Gamma$ is dense in $X$ and the stabilizer of $X$ in $G$ is finite. \\
  By \ref{algebraicweaksocletheorem}, for some $a \in G(K_1)$,  $a+X \subset S_{K_0}(G)$. Let $\Gamma_0$ be the finitely generated subgroup such that $\Gamma$ is contained in the ($p'$-)divisible hull of $\Gamma_0$. Let $\Gamma_1 := \langle a , \Gamma_0\rangle  \cap S_{K_0}(G)(K_1)$. Then $\Gamma_1$ is a finitely generated subgroup, let $\Gamma '$ be the ($p'$)-divisible hull of $\Gamma_1$, then $\Gamma ' \subset S_{K_0}(G)(K_1)$. Note that $(a+\Gamma) \cap S_{K_0}(G)(K_1) \subset \Gamma '$.  We claim that $\Gamma_1 \cap (a+X)$ is dense in $a+X$: Let $O$ be any open set (over $K_1$) in $a+X$, then $O-a$ is an open set of $X$. As $\Gamma \cap X$ is dense in $X$, there is some element of  $X\cap \Gamma$ in $O-a$, $g$. Then $a+g \in (a+X) \cap (a+\Gamma) \subset S_{K_0}(G)$. One can check that $a+g \in \Gamma '$.
  So $\Gamma ' \cap (a+X)$ is dense in $a+X$, and of course the stabilizer of $a+X$ in $S_{K_0}(G)$ is finite. Hence by our assumptions, there is $H_0$ some semiabelian subvariety of $S_{K_0}(G) $, a semiabelian variety $ S_0$ over $K_0$,  an irreducible subvariety $X_0$ of $S_0$ defined over $K_0$ and an isogeny $h$ from $H_0$ to $S_0$ such that $h^{-1}(X_0) = b+a+X$ for some $b\in S_{K_0}G(K_1)$.  This also gives the result (Mordell-Lang) for $G$ itself. \end{proof}

\medskip 

\noindent{\em Remark. } This reduction is a purely algebraic result, but our proof goes through model-theory and once one has the differential setting and the weak socle theorem, can be deduced quite easily.  To our knowledge this reduction as stated is new and we do not know if it has  a direct algebraic proof. 
 
The same question arises for the reduction of Mordell-Lang for semiabelian varieties  to the case of abelian varieties of trace zero  which we present in the next section.


\subsection{Function field Mordell-Lang for abelian varieties implies function field Mordell-Lang  for semiabelian varieties}


\medskip 
The material in this section is close to Hrushovski's original proof via the weak socle theorem, but replacing the use of one-basedness by the assumption that (function field) Mordell-Lang is true for abelian varieties. More precisely, 
we  show how to derive function field Mordell-Lang for semiabelian varieties from: 

\medskip 

\noindent{\bf Mordell-Lang for trace zero  abelian varieties:}  Let $K_0 <K_1$ be two algebraically closed fields, 
$A$ an abelian variety  over $K_1$ with $K_0$-trace zero, $X$ an irreducible subvariety of $A$  over $K_1$, and $\Gamma \subset G(K_1)$, a finite rank subgroup. Suppose that $\Gamma \cap X$ is Zariski dense in $X$ and that the stabilizer of $X$ in $G$ is finite. Then $X= \{a\}$ for some $a \in A(K_1)$. 


\begin{proposition} In all characteristics, function field Mordell-Lang for trace zero abelian varieties implies function field  Mordell-Lang for all semiabelian varieties.\end{proposition}
\begin{proof} 
First, by Proposition \ref{reducingtoalgebraicsocle}, we can suppose that $G = H_0 +A$, where $H_0$ is a semiabelian variety isogenous to one defined over $K_0$ and $A$ is an abelian variety with $K_0$-trace $0$. 

Again we pass to the model theoretic differential setting as described in the previous section, and so we have  $K$  a (sufficiently saturated) separably closed field (of non-zero finite degree of imperfection) in characteristic $p$  and a differentially closed field  in characteristic $0$;  $k$ denotes the constant field of $K$, and  $G= H_0 + A$, where $H_0$ is isogenous to a 
semiabelian variety over $k$, and $A$ has $k$-trace zero (by Lemmas \ref{descending} and \ref{lineardisjointness}). We have $\Gamma \subset G(K)$, $X$ an irreducible subvariety of $G$ with finite stabilizer in $G$, such that $X\cap G^\sharp$ is Zariski dense in $X$, and also
$\Gamma \cap X$ is dense in $X$.  As in Claim \ref{densityoftype}, there is a complete type $q_X$ in $X\cap G^\sharp $, with finite stabilizer, and which is Zariski dense in $X$, that is, 
$\overline{q_X(K)} = X$.

Remark that because of the ``socle-like'' form of $G$, 
$G^\sharp = {H_0}^\sharp + A^\sharp$  is already an almost pluri-minimal group, and so $S(G^\sharp) = G^\sharp$ (see Lemma 
\ref{lineardisjointness}), and  $H_0^\sharp$  and $A^\sharp$  are orthogonal.  By orthogonality, it follows that there are two complete orthogonal types 
$p_{H_0}$ in ${H_0}^\sharp$, and $p_A$  in $A^\sharp$ such that $q_X = p_{H_0} + p_A$. We let $V := \overline{p_{H_0}(K)} $ and $W:= \overline{p_A(K)}$.

\begin{claim} $X = {V+W} $ and the stabilizer of $W$ in $A$ is finite. \end{claim}
\begin{prclaim}  
First we show that  $V+W \subset X$. Fix any $b$ realizing the type $p_{H_0}$ and consider the set $\{c \in G : b+c  \in X\}$. This is closed in $G$, contains $p_A(K) $, hence contains 
$W= \overline{p_A(K)}$. So, for any $b$ realizing the type $p_{H_0}$, $b+W \subset X$. Now consider the set $\{b \in G : b+ W \subset X\}= \bigcap_{w \in W} X-w $. This is an intersection of closed sets, so it
is  closed, and it contains ${p_{H_0}(K)}$, so contains also $V= \overline{p_{H_0}(K)}$.  Hence $V+ W \subset X = \overline{{p_{H_0}(K)} + {p_A(K)}}$. 
Now in fact $V+W$ is closed: indeed  if $f$ is the natural map from $H_0\times A$ onto $H_0 + A_0$, as $H_0 $ and $A_0$ have finite intersection, $f$ is an isogeny of algebraic 
groups, hence it is closed, and $f(V\times W) = V + W$ is closed. \\
It follows that $X = V+W$.\\
Let $a \in A(K)$ be such that  $a+ W= W$. Then $a +X = a +V+W= V+(a+W) = V+ W = X$, 
so $a$ also stabilizes $X$.\end{prclaim}

\smallskip 

Now let $\pi : G \to G/H_0$. As $H_0 \cap A$ is finite, $\pi$ restricted to $A$, which we denote $\pi_A$,  is an isogeny. Let $A' := \pi_A(A)$, then $A'$ is also an abelian variety, and $\pi_A$  is  a closed morphism, so $W' := \pi_A (W)$ is closed in $A'$. \\

\begin{claim} The stabilizer of $W'$ in $A'$ is finite. \end{claim}
\begin{prclaim} This follows from the fact  that $\pi_A$ is an isogeny. Indeed, let
 $b \in Stab_{A'} W'$, $b = \pi_A (a)$; then $a \in Stab_G W + Ker\ \pi_A$:  
$\pi_A (a+W) = b + W' = W'$. So $a+W \subset {\pi_A}^{-1}(W') = W + Ker\  \pi_A = a_1+W \cup \ldots \cup a_n+W$. By irreducibility of $W$, for some $i$, $a+W = a_i +W$, 
so $a-a_i \in Stab_G W$, which is finite.  \end{prclaim}

\begin{claim} There is a finite rank subgroup in $A'(K)$, ${\Gamma}'$,  such that ${\Gamma}' \cap W'$ is Zariski dense in $W'$. \end{claim}
\begin{prclaim}  Let ${\Gamma}' := \pi_A (\Gamma)$. Then we claim that 
$\overline{{\Gamma}' \cap W'} = W'$. Indeed if $F$ is a closed set in $G/H_0$ such that 
$F \supset {\Gamma}'  \cap W' $, then $\pi^{-1}(F) \supset (W + H_0) \cap \Gamma  \supset X \cap \Gamma$. And as $\pi^{-1}(F)$ is closed, $\pi^{-1}(F)  \supset \overline{X \cap \Gamma } = X$.
Then $\pi(\pi^{-1}(F)) = F \supset \pi(X) = \pi(V+W)= \pi (W) = \pi_A (W)  = W'$. \end{prclaim}

\medskip

As the abelian  variety $A'$ is isogenous to $A$, it also has $k$-trace zero. By Mordell-Lang applied to $A'$, $W'$ and ${\Gamma}' $, $W' = \{a\}$.  It follows that $W$ itself is also reduced to a point, $W = \{b\}$. Hence $q_X  = p_{H_0} + b $. So the
Zariski closure of $q_X(K)$, $X$, is equal to $V +b$. Hence $X-b \subset V \subset H_0$.  As $H_0$ is isogenous  to some $S_0$ over the algebraically closed field $K_0$, let $h$ be the isogeny from $S_0\times_{K_0} K $ onto $H_0$ and let $X_0$ be  be the Zariski closure of 
$h^{-1} (p_{H_0}(K))$.  Then as required $X = h(X_0) +b$ , and as $h^{-1} (p_{H_0}(K))$ must be contained in ${S_0}^\sharp = S_0(k)$,  $X_0$ is defined over $k$. By linear disjointness of $k$ and $K_1$ over $K_0$, it is defined over $K_0$, as required. 
\end{proof}

\section{$G^\sharp$ versus $S(G^\sharp)$} 

We study the relationship between $G^{\sharp}$ and $S(G^{\sharp})$. We know from \cite{BBP1} that in the general semiabelian case, $G^{\sharp}$ and more generally the $\sharp$-functor may not be so well-behaved.  Here we point out ways in which the socle $S(G^{\sharp})$ is better behaved.

\subsection{Characteristic $p$ case, quantifier elimination for the socle of $G^\sharp$}\label{Gsharpincharp}

Let $G$ be a semiabelian variety over $K$ (sufficiently saturated) separably closed of characteristic $p>0$ and let $k:= K^{p^\infty}$ denote the field of constants.  In this section we will see that the socle of $G^\sharp$ recovers all the good properties of $A^\sharp$ when $A$ is an abelian variety, properties  that can be lost  for  semiabelian varieties: in particular  quantifier elimination and  relative Morley rank (see \cite{BBP1}, \cite{BBP2}). 

    Recall the notation from the previous sections: $S(G^\sharp) = (S_{k}(G))^\sharp = {G_0}^\sharp + {A_G}^\sharp$, where $G_0$ is the maximal  closed connected algebraic subgroup of $G$ which is isogenous to  $H \times_{k} K$, for some $H$ defined 
over $k$  ; $A_G$ is the maximal  abelian subvariety of $G$ with $k$-trace zero.\\
Furthermore, $G_0 $ and $A_G$ have finite intersection, and ${G_0}^\sharp$ and ${A_G}^\sharp$ are orthogonal.

 The first ``good'' properties of the socle of $G^\sharp$ follow easily from its definition.

\begin{proposition}\label{morleyrank}
1.  Suppose that $G$ is defined over $K_1<K$ where $K_1$ is the separable closure of ${{\mathbb F}_p}^{alg}(t)$.
Then $S(G^\sharp )(K_1)$ is contained in the torsion of $G$.\\
2. $S(G^\sharp)$ has finite relative Morley Rank.  
\end{proposition}
\begin{proof} 1.  Let $K_0$ denote ${{\mathbb F}_p}^{alg}$, the field of constants of the separably closed field $K_1$. Then  without loss of generality, $K_0 < k <K$, $K_0<K_1<K $ and $K_1$ and $k =C(K)$ are linearly disjoint over $K_0 = C(K_1)$.
  
As we noted in Lemma \ref{descending},  as $G$ is defined over $K_1$ then both the $K_0$-algebraic socle of $G$, $S_{K_0 }(G)$,  and the model theoretic socle 
$S(G^\sharp)= (S_{k}(G\times_{K_1} K))^\sharp$ are also defined over $K_1$, as are all their connected type-definable (or connected  closed,  in the algebraic case) subgroups.  
Also if $(S_{k}(G))^\sharp = {(G_0 \times_{K_1} K) }^\sharp + {(A_G\times_{K_1}K)}^\sharp$ as above, then $G_0$ is also the biggest closed connected algebraic subgroup of $G$ which is isogenous to $H \times_{K_0} K_1$, for some $H$ defined over $K_0$
 and $A_G$ is the biggest abelian subvariety with $K_0$-trace zero. \\
Then  $S(G^\sharp)(K_1)  =  {G_0}^\sharp (K_1) + {A_G}^\sharp (K_1)$. \\
By the result for abelian varieties (\cite{Roessler-torsion}), ${A_G}^\sharp (K_1)$ is contained in the torsion of $A_G$.  As $G_0$ is isogenous to some $S_0$ defined over $K_0$, the field of constants of $K_1$,  it follows that ${G_0}^\sharp (K_1)$ is definably isogenous to $S_0^\sharp(K_1) = S_0(K_0) = S_0({\Fp}^{alg})$  which is exactly the group of torsion elements of $S_0$.\\
2. By definition of the model theoretic socle, $S(G^\sharp) \subset acl (F\cup  Q_1\cup \ldots \cup Q_k)$ for some finite set $F$ and some (pairwise orthogonal) minimal types $Q_i$ in $G^\sharp$, and is of the form $B_{Q_1} +\ldots  + B_{Q_n}$, where each $B_{Q_j}$ is almost minimal with respect to $Q_j$. Then each
$B_{Q_j}$ has finite relative  Morley rank (see for example \cite{BBP1}, section 2.3 and Fact 3.8), hence so does $S(G^\sharp)$. \end{proof}

\medskip 

We will now show that the induced structure on the socle, in contrast with $G^\sharp$ itself, always has quantifier elimination. For precise definitions and basic facts about induced structures, see for example \cite{BBP2}, where the result is proved when $G$ is an abelian variety (it was proved there in the case of degree of imperfection $1$, but the proof is also valid for any non-zero finite degree of imperfection). A counterexample to quantifier elimination of the induced structure on $G^\sharp$ is given in
\cite{OmarAziz}: it is shown there that in the canonical example  given in \cite {BBP2}(Proposition 4.9 and Corollary 4.14) of  a semiabelian $G$ such that  $G^\sharp$ does not have relative Morley rank (taking the abelian part of $G$ to be an elliptic curve), then $G^\sharp$ does not eliminate quantifiers.

\medskip 

We will need a general lemma about quantifier elimination.

\begin{lemma} \label{QEisogeny}
In an ambiant, stable and sufficiently saturated structure, let $\Gamma$ and $\Delta$ be type-definable groups, and $f:\Gamma \to \Delta$ a relatively definable isogeny (i.e. surjective with finite kernel). Assume that $\Gamma$, with the structure induced by relatively definable subsets, has quantifier elimination. Then the same holds for $\Delta$. 
\end{lemma}
\begin{proof}
Write $\Gamma=\bigcap_i \Gamma_i$ as the intersection of a decreasing family of definable groups. We may assume that $f$ is the restriction to $\Gamma$ of a definable group homomorphism of domain $\Gamma_0$, denoted by $\tilde f$. Since $\hbox{Ker}(f)$ is definable, because finite, we know by compactness that there is some $i_0$ such that $\hbox{Ker}(f)=\hbox{Ker}(\tilde f) \cap \Gamma_{i_0}$.\\
Now consider a relatively definable  subset $D\subseteq \Delta^{r+1}$, and $\hbox{pr}_\Delta:\Delta^{r+1} \to \Delta^r$ a projection map; we have to prove that $\hbox{pr}_{\Delta}(D)$ is relatively definable. Let us denote by $f^r:\Gamma^r \to \Delta^r$ the cartesian power of $f$, then $\hbox{pr}_{\Delta}(D)=f^r(\hbox{pr}_{\Gamma}((f^{r+1})^{-1}(D)))$. By quantifier elimination for $\Gamma$, we know that $F:=\hbox{pr}_{\Gamma}((f^{r+1})^{-1}(D))$ is relatively definable in $\Gamma^r$; write it $F=\tilde F \cap \Gamma^r$ with $\tilde F$ definable.\\
Now 
$f^r (F) = \{(d_1,\ldots , d_r) \in \Delta^r; \exists g_1\ldots  \exists g_r \in \Gamma,\bigwedge_j f(g_j)=d_j \wedge (g_1,\ldots,g_r)\in \tilde F\}$ is contained in $X = \{(d_1,\ldots , d_r) \in \Delta^r; \exists g_1 \ldots  \exists g_r \in \Gamma_{i_0} ,\bigwedge_j \tilde f(g_j)=d_j \wedge (g_1,\ldots,g_r)\in \tilde F\}$.

Note that $X$ is a relatively definable subset of $\Delta^r$. If $(d_1,\ldots , d_r) \in X$, then for each $j$, 
$ d_j = \tilde f (\tilde g_j)$, with $\tilde g_j \in \Gamma_{i_0}$ and $(\tilde g_1,\ldots,\tilde g_r) \in \tilde F$; but since $f:\Gamma \to \Delta$ is surjective, we also have $d_j=f(g_j)$ with $g_j\in \Gamma$. So for each $j$, $\tilde g_j-g_j \in \hbox{Ker}(\tilde f)\cap \Gamma_{i_0} =\hbox{Ker}(f) \subset \Gamma$. Hence in fact $\tilde g_j \in \Gamma$ and $f^r (F) = X$, so $\hbox{pr}_{\Delta}(D)$  is relatively definable.
\end{proof}

\medskip

We are now back to $G$ defined over $K$. To simplify  notation, we now denote  $S(G) $ by $S$, $G_0$ by $H$ and $A_G$ by $A$ in the rest of this  section, hence $S(G^\sharp)= S^\sharp = H^\sharp + A^\sharp $. 

\begin{proposition} \label{QEsocle} The type-definable group $S^\sharp$, with the induced structure from $G$, has quantifier elimination.\end{proposition}
\begin{proof}

 We need to prove that, in $S^\sharp$, the projection of a relatively definable subset is still relatively definable. 
  
  \begin{claim} Both $H^\sharp$ and $A^\sharp$ are relatively definable in $S^\sharp$.\end{claim}
  \begin{prclaim} This  follows directly from the fact that $S^\sharp$ has finite relative Morley rank:  By \cite{BBP1} Lemma 3.6,  $H^\sharp$ is the connected component of $S^\sharp \cap H$. As $S^\sharp$ has finite relative Morley rank, it follows that $H^\sharp$ has finite index in $S^\sharp \cap H$, hence $H^\sharp$ is a finite union of coset of the relatively definable group $S^\sharp \cap H$.  Similarly for $A^{\sharp}$. 

 There is an alternative   short argument to show that the index is finite, which only uses the fact that $H\cap A$ is finite:  Consider any coset of $H^\sharp $ in $S^\sharp \cap H$, $s+H^\sharp$. As $s \in S^\sharp\cap H$, $s = h+a$, with $h \in H^\sharp$ and $a \in A^\sharp \cap H$,  
so $s+H^\sharp = a +H^\sharp$, and $a \in A \cap H$, which is finite. Same argument for $A^\sharp$. \end{prclaim}


 
Since the multiplication map gives an isogeny $H^\sharp \times A^\sharp \to H^\sharp+A^\sharp=S^\sharp$, we know by Lemma \ref{QEisogeny} that it suffices to prove quantifier elimination for $H^\sharp \times A^\sharp$. The semiabelian  variety $H$ is isogenous to $B \times_k K$ for some $B$ defined over the  field of constants $k$, hence $H^\sharp$ is isogenous to $B^\sharp = B(k)$. As the field $k$ is a pure algebraically closed subfield  in $K$, $B(k)$ has quantifier elimination, and $H^\sharp$ also by Lemma \ref{QEisogeny} again. And we know from \cite{BBP2} that $A^\sharp$ has quantifier elimination (this is true for all abelian varieties).

Now, by orthogonality of $A^\sharp $ and   $ H^\sharp$, if $F$ is a relatively definable subset of $(H^\sharp \times A^\sharp)^{r}$, then $F = \bigcup_{1<i<l} R_i \times T_i$, where $R_i$ is a relatively definable subset of ${H^\sharp}^r$, and $T_i$ a relatively definable subset of ${A^\sharp}^r$. It follows easily that in $(H^\sharp \times A^\sharp)^{r}$ , the projection of any relatively definable subset is also relatively definable.\end{proof}




\subsection{Some examples}\label{calculations}

In the following, we consider fields $k \subset K$, with $k$ algebraically closed,  and $K$ separably closed, in arbitrary characteristic; $G$ is a semiabelian variety over $K$.

We first deal with the extreme cases. If $G$ is isogenous to a semiabelian variety over $k$, then $S_{k}(G)=G$. If $G$ is an abelian variety, $S_{k}(G)=G$, as we already noted in lemma \ref{descending}. It is a consequence of the decomposition of abelian varieties into a sum of simple abelian varieties, hence a consequence  of the Poincar\'e reducibility Theorem. And in fact,  the gap between $G$ and $S_{k}(G)$ ``witnesses''  the defect in Poincar\'e's reducibility  for semiabelian variety, up to descent to $k$.

\smallskip

Let $G$ be an extension of an abelian variety $A$ by a torus $T=\Gm^s$, all defined over  $K$. Let us recall that such extensions are parametrized, up to isomorphism, by $Ext(A,T)\simeq \hat A^s$, where $\hat A$ is the dual abelian variety of $A$. We will say that the extension $0 \to T \to G \to A \to 0$ is {\em almost split} if there is a semiabelian subvariety $B$ of $G$ such that $G=T+B$ and $T\cap B$ is finite. Note that in this case, $B$ must be isogenous to $A$, and it follows that $G$ is almost split if and only if $G$ is isogenous to the direct product $T \times A$ (which corresponds to the neutral element in $Ext(A,T)$).

Recall that we say $G$ {\em descends } to $k$ if $G$ is {\em isomorphic} to  a semiabelian variety $H \times_{k} K$, where $H$ is defined over $k$.  In characteristic $0$, if $G$ is isogenous to some $H \times_{k} K$, where $H$ is defined over $k$, then $G$ descends to $k$ but in characteristic $p$, this is not always true.

The following facts are undoubtly  quite classical. For lack of an easy reference, we reprove them in  Appendix II. 
\begin{fact} \label{Ext}
1. $G$ is almost  split if and only if $G$ is a torsion point in $Ext(A,T)$.\\
2. Assume that $A$ is defined over $k$. If $G$ is isogenous to $H \times_k  K$, where $H$ is a semiabelian variety over $k$, then $G$ descends to $k$.
\end{fact}

Now let us summarize the situation.
\begin{proposition} \label{summary}
{ 1.} Assume that $G$ is almost split. Then $S_{k}(G)=G$.\\
{2.} Conversely:\par 
(a)  If $A$ has $k$-trace $0$, $S_{k}(G)=G$ if and only if $G$ is almost split.  \par 
(b) If $A$ is isogenous to an abelian variety over $k$, then $S_{k}(G) = G$ if and only if $G$ is isogenous to a semiabelian variety over $k$. If we assume moreover that $A$ is simple, then $S_{k}(G)\neq G$ if and only if $S_{k}(G) = T$. \par
(c) If $A$ descends to $k$, then $S_{k}(G)=G$ if and only if $G$ descends to $k$. 
\end{proposition}

\begin{remark}
{\em The previous proposition enables  us to cover all the cases: If the extension is given by the map $f:G\to A$, and if $A=A_1+\ldots+A_r$ is the decomposition of $A$ into a sum of simple abelian varieties, then $G=G_1+\ldots+G_r$, where $G_i=f^{-1}(A_i)$ and $S_{k}(G)=S_{k}(G_1)+\ldots+S_{k}(G_r)$. And, for $A_i$ a simple abelian variety, either $A_i$ has $k$-trace $0$, or $A_i$ is isogenous to some $B \times _k K$, for $B$ an abelian variety over $k$.}

\end{remark}

\begin{proof}
\begin{enumerate}
\item If $G=T+B$, where $B$ is an abelian subvariety of $G$, we know from the extreme cases that $S_{k}(G)=S_{k}(T)+S_{k}(B)=T+B=G$.
\item
\begin{enumerate}
\item Assume that $A$ has $k$-trace $0$ and $S_{k}(G)=G_0+A_G=G$. Since $A$ has $k$-trace $0$, $G_0$ has a trivial abelian part, i.e. $G_0=T$. Hence $G$ is almost split.
\item One direction is one of the extreme cases. Now assume that $G$ is not isogenous to a semiabelian variety over $k$, and write $S_{k}(G)=G_0+A_G$. The map $f:G \to A$ induces a map $A_G \to A$, with finite kernel. But $A_G$ has $k$-trace $0$, and $A$ is isogenous to an abelian variety over $k$, hence $A_G=0$. Now $S_{k}(G)=G_0\neq G$, since $G_0$ is isogenous to some $H_K$, for $H$ a semiabelian variety over $k$, contrary to $G$. If we assume furthermore that $A$ is simple, we know that $T \subseteq G_0$, and the abelian part of $G_0$ is $0$ since it can not be $A$. Hence $S_{k}(G)=T$.
\item We can apply (b) for the remaining implication: if $S_k(G)=G$, then $G$ is isogenous to a semiabelian variety over $k$. But using Fact \ref{Ext}.2, this  implies that $G$ descends to $k$.
\end{enumerate}
\end{enumerate}
\end{proof}

\noindent
Suppose now  that 
$K$ is differentially closed of characteristic $0$, or $K$ is a (sufficiently saturated) model of $SCF_{p,1}$, and let $k=C(K)$ be the field of constants. Let $0\to T \to G \to A \to 0$ be a semiabelian variety over $K$. We discussed in \cite{BBP1} the question of whether the induced sequence
$$
(\sharp) \quad 0 \to T^\sharp \to G^\sharp \to A^\sharp \to 0
$$  
is exact or not. It relates to the current question:

\begin{proposition}
If $S_k(G) =  G$, then  the sequence $(\sharp)$ is exact. 
\end{proposition}
\begin{proof}
Suppose that the sequence $(\sharp)$ is not exact. Since $S_k(G)^\sharp=S(G^\sharp)$, it suffices to prove that $S(G^\sharp) \neq G^\sharp$.\\
In characteristic $p$, we have shown (Proposition \ref{morleyrank}) that $S(G^\sharp)$ always has  finite relative Morley rank. We had shown in  \cite{BBP1} that $G^\sharp$ had finite relative Morley rank if and only if the sequence $\sharp$ is exact. \\
In characteristic $0$, we know that every connected definable subgroup of $S(G^\sharp)$ can be written as $H^\sharp$ for some semiabelian subvariety $H$ (see Corollary \ref{char.0}), and it is not the case for $G^\sharp$ (see Remark \ref{notexact}).
\end{proof}

Unfortunately, the only known cases at the moment where the sequence $(\sharp)$ is not exact are already covered by Proposition \ref{summary}: it is when $A$ descends to $k$ (and is ordinary) but $G$ does not descend to $k$.\\
For the other direction, there are  two different kinds of examples which show that the $\sharp$ sequence can be exact without $S_k(G) = G$:
\\
In characteristic $p$, take $A$ an abelian variety over $k$ with $p$-rank $0$ (which implies $A[p^\infty](K)=0$), and $G$ an extension of $A\times_k K$ by $T$ which does not descend to $k$, then the sequence $(\sharp)$ is exact (see Corollary 4.15 in \cite{BBP1}) but  $S_{k}(G)\neq G$.\\
In characteristic $0$, take $A$ an elliptic curve which does not descend to $k$, and $G$ a non almost split extenstion of $A$ by $T$. Following \cite{BBP1}, we know that $U_A$, the maximal unipotent $D$-subgroup of the universal extension $\tilde A$ of $A$ by  a vector group (which is endowed with its unique $D$-group structure), is $0$. But in this case, the map $U_G \to U_A$, induced by the map $G \to A$, is surjective, from which we deduce that the sequence $(\sharp)$ is exact (see Proposition 5.21 in \cite{BBP1}). But we know from Proposition \ref{summary} that $S_k(G)\neq G$.


\section{Appendix I: modularity}

The key argument in Hrushovski's proof of the Mordell-Lang conjecture was a model-theoretic result, more precisely a dichotomy result for Zariski types. The question whether one can go in the other direction and deduce this dichotomy result for abelian varieties from the Mordell-lang statement has been open since. More precisely the dichotomy result one wants to deduce has the  following form, where as before $K$ is a differentially closed field in characteristic  $0$ and a separably closed field  in characteristic $p$, and $k $ is the field of constants of $K$:

\smallskip

\noindent
\begin{it}
Let $A$ be an abelian variety over $K$ with $k$-trace $0$. Then
\newline
(i) $A^{\sharp}$ is $1$-based.
\newline
(ii) If moreover $A$ is simple, then $A^{\sharp}$ is  strongly minimal. 
\end{it}
\newline

\smallskip
Remark: The reader might think that this dichotomy  follows as in  \cite{Pillay-L} where the truth of Mordell-Lang is shown to be equivalent to the $1$-basedness of $\Gamma$ in the expansion $(K,+,\times, \Gamma)$ by the relevant finite rank subgroup $\Gamma$ of the relevant (semi)abelian variety. However we are working here with $A^{\sharp}$ with all of its  induced structure from the model $K$ of $DCF_0$ or $SCF_{p,1}$, so additional arguments are needed.\\

We are able to deduce the dichotomy statement from Mordell-Lang, but  only in the cases (from \cite{BBP2}) when we know how to prove Mordell-Lang  without going through Zariski   structures,  deducing it from Manin-Mumford.

Here also we will in fact deduce the statement  from the related Manin-Mumford statement:
\begin{theorem}[MM] Let $K$ be an algebraically closed field of characteristic $0$, and $G$ a commutative algebraic group over $K$; or let $k<K$ be algebraically closed fields of characteristic $p$, and $G$ an abelian variety over $K$ with $k$-trace $0$.\\
Let $X$ be an algebraic subvariety of $G$ defined over $K$.\\
Then there are a finite number of torsion points $t_1,\ldots,t_m$, and algebraic subgroups $G_1,\ldots,G_m$ over $K$, such that
$$
X(K)\cap G(K)_{torsion}=\bigcup_{i=1}^m t_i+G_i(K)_{torsion}.
$$ 
\end{theorem}

\begin{remark}
In characteristic $0$, the statement with this level of generality has been proved by Hindry in \cite{Hindry}. We will use it for $G$ an extension of the abelian variety $A$ by a vector group. In characteristic $p$, the statement was proved by Pink and R\"ossler in \cite{PinkRossler}. The hypothesis of $k$-trace $0$ is mandatory since for an abelian variety over $\Fp^{\rm{alg}}$, every $\Fp^{\rm{alg}}$-rational point is a torsion point.
\end{remark}

As in \cite{BBP2}, we will assume that $K_0=(\mathbb{C}(t))^{\rm{alg}}$ considered as a differential field with the natural extension of  $d/dt$  and $K=K_0^{\rm{diff}}$ its differential closure in characteristic $0$, and $K_0=K=\Fp(t)^{\rm{sep}}$, and $\mathcal U$ an $\aleph_1$-saturated elementary extension of $K$ in characteristic $p$.  In this situation, we can use the ``theorem of the Kernel'':
\begin{theorem}
Let $A$ be an abelian variety over $K_0$, of $C(K_0)$-trace $0$. Then every point in $A^\sharp(K)$ is a torsion point.
\end{theorem} 
\begin{remark}
The characterisic $0$ case is proved in \cite{BP}, and the characteristic $p$ case in \cite{Roessler-torsion}.
The trace $0$ hypothesis is only mandatory in characteristic $0$.
\end{remark}

\smallskip

Recall (from \cite{Pillaybook}) that a type-definable connected commutative group $G$ (in a sufficiently saturated model of a stable theory) is $1$-based if every relatively definable subset of $G^{n}$ is a Boolean combination of translates of relatively definable subgroups (over the same algebraically closed set of parameters as $G$). In the present situation, where $A^\sharp$ is defined over the algebraically closed set of parameters $K$, it suffices that every relatively $K$-definable subset of $A^{\sharp n}$ is is a Boolean combination of translates of relatively $K$-definable subgroups.\\
In the characteristic $p$ case, we proved in \cite{BBP2} that the structure $\mathcal A$, whose underlying set is $A^\sharp=A^\sharp(\mathcal U)$, and whose basic definable sets are the relatively $K$-definable subsets of the cartesian powers $A^{\sharp n}$, has quantifier elimination. It follows that $A^\sharp$ is one-based iff $\mathcal A$ is. But we are now in a definable setting (rather than type-definable), for which the saturation hypothesis is no longer relevant.  Since we also know from \cite{BBP2} that $A^\sharp(K)$, with the same basic definable sets as above, is an elementary substructure of $\mathcal A$, we are allowed to replace $A^\sharp (\mathcal U)$ by $A^\sharp(K)$.

\begin{proposition} \label{dichotomy}
Let $K_0$ and $K$ be as above, and $k=C(K)=C(K_0)$ the subfield of  constants.
Let $A$ be an abelian variety over $K_0$ with $k$-trace $0$. Then
\newline
(i) $A^{\sharp}$ is $1$-based.
\newline
(ii) If moreover $A$ is simple, then $A^{\sharp}$ is  strongly minimal. 
\end{proposition} 

\begin{proof}
\newline
(i) By replacing $A$ by a cartesian power, we just have to consider (relatively) definable subset of $A^\sharp$. By quantifier elimination in the ambiant structure, a definable set is a boolean combination of Kolchin-closed sets in characteristic $0$, or of $\lambda_n$-closed sets for some $n$ in characteristic $p$, so we can reduce to those cases.\\
In characteristic $0$, we use the description of $A^\sharp$ given in $\cite{BP}$, with the same notations: there is an extension $\pi : \bar A \to A$ of $A$ by a vector group, such that $\bar A$ is endowed with a unique $D$-group structure, and such that $\pi$ induces an isomorphism (in the category of differential algebraic groups) between $\bar A^\partial$ and $A^\sharp$. If $X$ is a Kolchin-closed set in $A^\sharp$, $\pi^{-1}(X)$ is a Kolchin-closed set in $\bar A^\partial$, which can be written as $Y\cap \bar A^\partial$ for some algebraic subvariety $Y$ of $\bar A$ since, in $\bar A^\partial$, the derivation coincides with a regular algebraic map. Furthermore, as $\pi$ has a torsion free kernel, it induces an isomorphism between torsion points of $\bar A$ and torsion points in $A$. Now we can use MM for $Y$ in $\bar A$: $Y(K) \cap \bar A(K)_{torsion}$ is of the required form, and so is its isomorphic image by $\pi$, which is $X(K) \cap A(K)_{torsion}=X(K)$.\\
In characteristic $p$, we use the description of $A^\sharp$ and its $\lambda_n$-closed relatively definable subsets given in Section 3 of \cite{BBP2}. There is an isogeny of abelian varieties over $K$, $\pi : A_n \to A$, and a definable map $\lambda_n:A(K) \to A_n(K)$, which induce an isomorphism from $A_n(K)$ to $p^n A(K) < A^\sharp$, with inverse isomorphism induced by $\lambda_n$. Furthermore, for $X$ a relatively definable $\lambda_n$-closed subset of $A^\sharp$, there is an algebraic subvariety $X_n$ of $A_n$ such that $X=\lambda_n^{-1}(X_n)$. Now we use MM for $X_n$ in $A_n$ (which is still an abelian variety with $C(K_0)$-trace $0$, see \cite{conrad} for the basic properties of the trace): $X_n(K) \cap A_n(K)_{torsion}$ is of the required form, and so is $X$, its image by $\pi_n$.
\newline 
(ii) If $A$ is simple, then $A^{\sharp}$ is $g$-minimal, i.e. has no infinite (relatively) definable proper subgroup. So by part (i), every (relatively) definable subset of $A^{\sharp}$ is finite or cofinite. 
\end{proof}

\section{Appendix II: $Ext(A,T)$}

We will work with two fields $k\subset K$, $k$ algebraically closed, in arbitrary characteristic. We use results from \cite{Serre} about $Ext(A,T)$, the moduli space for extensions of $A$ by $T$, for given abelian variety $A$ and torus $T=\Gm^s$ over $k$, that we recall now.

The mapping $(A,T)\mapsto Ext(A,T)$ can be made into a bifunctor, covariant in $T$ and contravariant in $A$, in the following way.\\
If $G\in Ext(A,T)$ and $f:T \to T'$ is a homomorphism, $f_*(G)$ is characterized as the unique element in $Ext(A,T')$ such that there is a homomorphism $u:G\to f_*(G)$ making the diagram commute:
~\\
\begin{center}
\begin{pspicture}(8,1)
\rput(0,1){\rnode{A}{$0$}}
\rput(2,1){\rnode{B}{$T$}}
\rput(4,1){\rnode{C}{$G$}}
\rput(6,1){\rnode{D}{$A$}}
\rput(8,1){\rnode{E}{$0$}}
\rput(0,0){\rnode{F}{$0$}}
\rput(2,0){\rnode{G}{$T'$}}
\rput(4,0){\rnode{H}{$f_*(G)$}}
\rput(6,0){\rnode{I}{$A$}}
\rput(8,0){\rnode{J}{$0$}}
\psset{arrows=->,nodesep=3pt,shortput=tablr,linewidth=0.1pt}
\ncline{A}{B}
\ncline{B}{C} 
\ncline{C}{D}
\ncline{D}{E}
\ncline{F}{G}
\ncline{G}{H}
\ncline{H}{I}
\ncline{I}{J}
\ncline{B}{G}<{$f$}
\ncline{C}{H}<{$u$}
\ncline{D}{I}<{id} 
\end{pspicture}
\end{center}
~\\
Similarly, for a homomorphism $g:A'\to A$, $g^*(G)$ is characterized as the unique element in $Ext(A',T)$ such that there is a homomorphism $u:g^*(G)\to G$ making the diagram commute:
~\\
\begin{center}
\begin{pspicture}(8,1)
\rput(0,1){\rnode{A}{$0$}}
\rput(2,1){\rnode{B}{$T$}}
\rput(4,1){\rnode{C}{$g^*(G)$}}
\rput(6,1){\rnode{D}{$A'$}}
\rput(8,1){\rnode{E}{$0$}}
\rput(0,0){\rnode{F}{$0$}}
\rput(2,0){\rnode{G}{$T$}}
\rput(4,0){\rnode{H}{$G$}}
\rput(6,0){\rnode{I}{$A$}}
\rput(8,0){\rnode{J}{$0$}}
\psset{arrows=->,nodesep=3pt,shortput=tablr,linewidth=0.1pt}
\ncline{A}{B}
\ncline{B}{C} 
\ncline{C}{D}
\ncline{D}{E}
\ncline{F}{G}
\ncline{G}{H}
\ncline{H}{I}
\ncline{I}{J}
\ncline{B}{G}<{id}
\ncline{C}{H}<{$u$}
\ncline{D}{I}<{$g$} 
\end{pspicture}
\end{center}
~\\
We will also use that the existence of a commutative diagram
~\\
\begin{center}
\begin{pspicture}(8,1)
\rput(0,1){\rnode{A}{$0$}}
\rput(2,1){\rnode{B}{$T$}}
\rput(4,1){\rnode{C}{$G$}}
\rput(6,1){\rnode{D}{$A$}}
\rput(8,1){\rnode{E}{$0$}}
\rput(0,0){\rnode{F}{$0$}}
\rput(2,0){\rnode{G}{$T'$}}
\rput(4,0){\rnode{H}{$G'$}}
\rput(6,0){\rnode{I}{$A'$}}
\rput(8,0){\rnode{J}{$0$}}
\psset{arrows=->,nodesep=3pt,shortput=tablr,linewidth=0.1pt}
\ncline{A}{B}
\ncline{B}{C} 
\ncline{C}{D}
\ncline{D}{E}
\ncline{F}{G}
\ncline{G}{H}
\ncline{H}{I}
\ncline{I}{J}
\ncline{B}{G}<{$f$}
\ncline{C}{H}<{$u$}
\ncline{D}{I}<{$g$} 
\end{pspicture}
\end{center}
~\\
is equivalent to the equality $f_*(G)=g^*(G')$ in $Ext(A,T')$.\\
These constructions are used to put a commutative group law on $Ext(A,T)$, with the trivial extension $T\times A$ as neutral element. With these group laws, the maps $f_*$ and $g^*$ are now homomorphisms. We will describe below the multiplication by $n$ for this group law. Now, the functor $K\mapsto Ext(A\times_k K,T\times_k K)$, from field extensions of $k$ into groups, is represented by $\hat A^s$, where $\hat A$ is the dual abelian variety of $A$. In other words, the extensions (implicitly: defined over $K$) of $A\times_k K$ by $T\times_k K$, up to isomorphism, are parametrized by $\hat A^s(K)$. 

Let $0\to T \to G \stackrel{f}{\to} A \to 0$ be an element in $Ext(A,T)$. Via the map $f^n$, $G^n$ is an element of $Ext(A^n,T^n)$ ($n\ge 1$). For any algebraic group $H$, let $d_{nH} : H \to H^n$ be the diagonal map and $m_{nH} : H^n \to H$ be the multiplication map. Then, in $Ext(A,T)$, $[n]G=d_{nA} ^*m_{nT*}G^n=m_{nT*}d_{nA}^*G^n$.\\
Now consider the commutative diagram
~\\
\begin{center}
\begin{pspicture}(8,2)
\rput(0,2){\rnode{A}{$0$}}
\rput(2,2){\rnode{B}{$T$}}
\rput(4,2){\rnode{C}{$G$}}
\rput(6,2){\rnode{D}{$A$}}
\rput(8,2){\rnode{E}{$0$}}
\rput(0,1){\rnode{F}{$0$}}
\rput(2,1){\rnode{G}{$T^n$}}
\rput(4,1){\rnode{H}{$G^n$}}
\rput(6,1){\rnode{I}{$A^n$}}
\rput(8,1){\rnode{J}{$0$}}
\rput(0,0){\rnode{K}{$0$}}
\rput(2,0){\rnode{L}{$T$}}
\rput(4,0){\rnode{M}{$G$}}
\rput(6,0){\rnode{N}{$A$}}
\rput(8,0){\rnode{O}{$0$}}
\psset{arrows=->,nodesep=3pt,shortput=tablr,linewidth=0.1pt}
\ncline{A}{B}
\ncline{B}{C}
\ncline{C}{D}^{$f$}
\ncline{D}{E}
\ncline{F}{G}
\ncline{G}{H}
\ncline{H}{I}^{$f^n$}
\ncline{I}{J}
\ncline{B}{G}<{$d_n$}
\ncline{C}{H}<{$d_n$}
\ncline{D}{I}>{$d_n$} 
\ncline{K}{L}
\ncline{L}{M} 
\ncline{M}{N}^{$f$}
\ncline{N}{O}
\ncline{G}{L}<{$m_n$}
\ncline{H}{M}<{$m_n$}
\ncline{I}{N}>{$m_n$}
\end{pspicture}
\end{center}
~\\
The composition on the columns is $[n]$. From the upper half of the diagram, we get that $d_{nA}^*G^n=d_{nT*}G$ in $Ext(A,T^n)$. From the lower half, we get that $m_{nA}^*G=m_{nT*}G^n$ in $Ext(A^n,T)$.
Hence we get 
$$[n]G=d_{nA}^*m_{nT*}G^n=d_{nA}^*m_{nA}^*G=(m_{nA}d_{nA})^*G=[n]_A^*G$$
and also 
$$[n]G=m_{nT*}d_{nA}^*G^n=m_{nT*}d_{nT*}G=(m_{nT}d_{nT})_*G=[n]_{T*}G.$$

\begin{proposition} \label{torsionsplit}
$G$ is an almost split extension of $A$ by $T$ if and only if it is a torsion point in $Ext(A,T)$.
\end{proposition}

\begin{proof}
One direction uses the general fact that $G$ is isogenous to $[n]G$: Indeed, since $[n]G=[n]_{T*}G$, 
there is an homomorphism $u$ such that the following diagram commutes
~\\
\begin{center}
\begin{pspicture}(8,1)
\rput(0,1){\rnode{A}{$0$}}
\rput(2,1){\rnode{B}{$T$}}
\rput(4,1){\rnode{C}{$G$}}
\rput(6,1){\rnode{D}{$A$}}
\rput(8,1){\rnode{E}{$0$}}
\rput(0,0){\rnode{F}{$0$}}
\rput(2,0){\rnode{G}{$T$}}
\rput(4,0){\rnode{H}{$[n]G$}}
\rput(6,0){\rnode{I}{$A$}}
\rput(8,0){\rnode{J}{$0$}}
\psset{arrows=->,nodesep=3pt,shortput=tablr,linewidth=0.1pt}
\ncline{A}{B}
\ncline{B}{C} 
\ncline{C}{D}^{$f$}
\ncline{D}{E}
\ncline{F}{G}
\ncline{G}{H}
\ncline{H}{I}
\ncline{I}{J}
\ncline{B}{G}<{$[n]_{T}$}
\ncline{C}{H}<{$u$}
\ncline{D}{I}<{id} 
\end{pspicture}
\end{center}
~\\
Clearly, $u$ is an isogeny because $[n]_{T}$ and $id$ are. We could have made the same argument using that $[n]G=[n]_A^*G$. Hence, if $[n]G= T \times A$, $G$ is almost split.\\
Conversely, if $\phi: G \to T \times A$ is an isogeny, $\phi$ must map the linear part of $G$ into the linear part of $T\times A$, hence it induces a commutative diagram
~\\
\begin{center}
\begin{pspicture}(8,1)
\rput(0,1){\rnode{A}{$0$}}
\rput(2,1){\rnode{B}{$T$}}
\rput(4,1){\rnode{C}{$G$}}
\rput(6,1){\rnode{D}{$A$}}
\rput(8,1){\rnode{E}{$0$}}
\rput(0,0){\rnode{F}{$0$}}
\rput(2,0){\rnode{G}{$T$}}
\rput(4,0){\rnode{H}{$T\times A$}}
\rput(6,0){\rnode{I}{$A$}}
\rput(8,0){\rnode{J}{$0$}}
\psset{arrows=->,nodesep=3pt,shortput=tablr,linewidth=0.1pt}
\ncline{A}{B}
\ncline{B}{C} 
\ncline{C}{D}
\ncline{D}{E}
\ncline{F}{G}
\ncline{G}{H}
\ncline{H}{I}
\ncline{I}{J}
\ncline{B}{G}<{$\phi_{T}$}
\ncline{C}{H}<{$\phi$}
\ncline{D}{I}<{$\phi_A$} 
\end{pspicture}
\end{center}
~\\
Such a diagram implies that $\phi_{T*}G=\phi_A^* (T\times A)=T\times A$, since $\phi_A^*$ is a group homomorphism. Now, take $\theta :T \to T$ such that $\theta \circ \phi_T = [n]_T$ for some $n\ge 1$. We obtain $[n]G = [n]_{T*}G=\theta_* \phi_{T*} G =\theta_* (T\times A) = T \times A$, i.e. $G$ is a torsion point.
\end{proof}


\begin{proposition} \label{isogenousdescent} Recall that $A$ is defined over $k$. 
  Let $G \in Ext(A\times_k K,T\times_k K)$ be  defined over $K$. If $G$ is isogenous to some $H\times_k K$ for $H$
  a semiabelian variety over $k$, then $G$ descends to $k$.
\end{proposition}
\begin{proof}
Write $H\in Ext(A',T)$ for some abelian variety $A'$, necessarily over $k$. We assume that there is an isogeny $\phi :G \to H_K$; since $\phi$ must map the linear part of $G$ into the linear part of $H_K$, it induces isogenies $\phi_T:T_K \to T_K$ and $\phi_A:A_K \to A'_K$, such that the following diagram commutes
~\\
\begin{center}
\begin{pspicture}(8,1)
\rput(0,1){\rnode{A}{$0$}}
\rput(2,1){\rnode{B}{$T_K$}}
\rput(4,1){\rnode{C}{$G$}}
\rput(6,1){\rnode{D}{$A_K$}}
\rput(8,1){\rnode{E}{$0$}}
\rput(0,0){\rnode{F}{$0$}}
\rput(2,0){\rnode{G}{$T_K$}}
\rput(4,0){\rnode{H}{$H_K$}}
\rput(6,0){\rnode{I}{$A'_K$}}
\rput(8,0){\rnode{J}{$0$}}
\psset{arrows=->,nodesep=3pt,shortput=tablr,linewidth=0.1pt}
\ncline{A}{B}
\ncline{B}{C} 
\ncline{C}{D}
\ncline{D}{E}
\ncline{F}{G}
\ncline{G}{H}
\ncline{H}{I}
\ncline{I}{J}
\ncline{B}{G}<{$\phi_T$}
\ncline{C}{H}<{$\phi$}
\ncline{D}{I}<{$\phi_A$} 
\end{pspicture}
\end{center}
~\\ 
Since $A$ and $A'$ are abelian varieties over $k$, we know from Chow (\cite{Chow}) that there is an isogeny $\psi : A \to A'$ such that $\phi_A = \psi_K$ is obtained by base change. It follows that, in $Ext(A_K,T_K)$, $\phi_{T*}G=\phi_A^*H_K=\psi_K^* H_K=(\psi^*H)_K$. Since $\psi^*H$ is over $k$, it is a $k$-rational point of ${\hat A}^s$. Furthermore, there is an isogeny $\theta : T_K \to T_K$ such that $\theta \circ \phi_T=[n]_T$ for some integer $n\ge 1$. Note that $\phi_T$ and $\theta$ are actually defined over $k$. Hence $[n]G = [n]_{T*}G=\theta_* \phi_{T*} G \in {\hat A}^s (k)$, with $k$ algebraically closed, and then
$G\in {\hat A}^s(k)$, that is, $G$ descends to $k$.
\end{proof}

\begin{remark} In characteristic $0$, the argument is much simpler and more general (it works for any semiabelian variety): since there is no separability issue, we are allowed to take the quotient by the kernel of an isogeny. But in characteristic $p$, we have to restrict to particular cases, since we know examples of abelian varieties $A$ over $K$, which are isogenous to some $B_K$ for some $B$ over $k$, but such that $A$ does not descend to $k$ (see \cite{BBP1}).
\end{remark}

{\em Franck Benoist and Elisabeth Bouscaren 

Laboratoire de Math\'ematiques d'Orsay

Universit\'e Paris-Sud, CNRS

91405 Orsay, France.}

franck.benoist@math.u-psud.fr

elisabeth.bouscaren@math.u-psud.fr

\medskip

{\em Anand Pillay

Department  of Mathematics

University of  Notre Dame

Notre Dame, IN 46556 USA.}

Anand.Pillay.3@nd.edu


\begin{thebibliography}{99}

\bibitem{BBP1} F. Benoist, E. Bouscaren, and A. Pillay, Semiabelian
  varieties over separably closed fields, maximal divisible subgroups,
  and exact sequences, Journal of the Institute of Mathematics of
  Jussieu 15 (2016), no. 1, 29-69.

\bibitem{BBP2} F. Benoist, E. Bouscaren, and A. Pillay, On function field
  Mordell-Lang and Manin-Mumford, Journal of
  Mathematical Logic 16 (2016), no. 1. DOI: 10.1142/S021906131650001X

\bibitem{BP} D. Bertrand and A.Pillay,  A Lindemann-Weierstrass theorem for semiabelian varieties over function fields, Journal AMS 23 (2010), 491-533. 

\bibitem{bouscaren} E. Bouscaren, Proof of the Mordell-Lang conjecture, in {\em Model theory and Algebraic geometry},  E. Bouscaren Ed., Lecture Notes in
  Mathematics 1696, 2nd edition, Springer, 1999. 

\bibitem{BouscarenDelon} E. Bouscaren and F. Delon, Groups definable in separably closed fields, Transactions of the AMS 354 (2002),vol. 3, 945-966.

\bibitem{BD} E. Bouscaren and F. Delon, Minimal groups in separably closed fields,  J. Symbolic Logic 67 (2002), no. 1, 239-259.

\bibitem{Chow} W.L. Chow, Abelian varieties over function fields, Transactions of the AMS 78 (1955), 253-275. 

\bibitem{conrad} B. Conrad, Chow's K/k-image and K/k-trace, and the Lang-N\'eron theorem, Enseign. Math. (2) 52 (2006), no. 1-2, 37-108.

\bibitem{Hindry} M. Hindry, Autour d'une conjecture de Serge Lang, Invent. Math. 94 (1988), no. 3, 575-603.

\bibitem{Hindry2} M. Hindry, Introduction to abelian varieties and the Lang Conjecture, in {\em Model theory and Algebraic geometry}, E. Bouscaren Ed., Lecture Notes in   Mathematics 1696, 2nd edition, Springer, 1999.

\bibitem{Hrushovski} E. Hrushovski, The Mordell-Lang conjecture for function fields, Journal AMS 9 (1996), 667-690.

\bibitem{lascar} D. Lascar, $\omega$-stable groups, in {\em Model theory
  and algebraic geometry}, E. Bouscaren Ed., Lecture Notes in
  Mathematics 1696, 2nd edition, Springer, 1999.

\bibitem{MMP} D. Marker, M. Messmer, A. Pillay,  {\em Model theory of Fields}, Lecture Notes in Logic 5, Second edition, ASL-AK Peters, 2003.


\bibitem{OmarAziz} A. Omar Aziz, {\em Type definable stable groups and separably closed fields}, Ph. D. thesis, Leeds, 2012. 



\bibitem{Pillay-stability} A. Pillay, {\em An Introduction to Stability Theory}, Oxford University Press, 1983. 

\bibitem{Pillaybook} A. Pillay {\em Geometric Stability Theory}, Oxford University Press, 1996. 

\bibitem{Pillay-L} A. Pillay, The model-theoretic content of Lang's conjecture, {\em Model theory and algebraic geometry}, edited by E. Bouscaren, Lecture Notes in Mathematics 1696, Springer, 1996.

\bibitem{Pillay-ML} A. Pillay, Mordell-Lang in characteristic zero revisited, Compositio Math. 140 (2004), no. 1,  64-68. 

\bibitem{PinkRossler} R. Pink and D. R\"ossler, On $\psi$-invariant subvarieties of semiabelian varieties and the Manin-Mumford conjecture, J. Algebraic Geom. 13 (2004), no. 4, 771-798. 

\bibitem{Poizatbook} B. Poizat, {\em Stable groups}, Mathematical Surveys and Monographs, Vol. 87, American Mathematical Society, 2001.

\bibitem{Robinson} A. Robinson, Solution of a problem of Tarski, Fund. Math. 47 (1959), 179-204. 

\bibitem{Roessler-torsion} D. R\"ossler,  Infinitely $p$-divisible points on abelian varieties defined over function fields of characteristic $p>0$, Notre Dame Journal of Formal Logic 54 (2013),  no. 3-4, 579-589.

\bibitem{Serre} J. P. Serre, {\em Algebraic groups and class fields}, Springer, 1988. 

\bibitem{Wagnerbook} F. Wagner, {\em Stable groups}, Cambridge University Press, 1997.

\bibitem{Ziegler} M. Ziegler, Separably closed fields with Hasse derivations, J. Symbolic Logic 68 (2003), no. 1, 311-318. 

\end{thebibliography}
\end{document}